\documentclass[a4paper,12pt,oneside]{article}
\usepackage[english]{babel}
\usepackage[T1]{fontenc} 
\usepackage[utf8]{inputenc}
\usepackage{amsthm}
\usepackage{bbm}
\usepackage{amsmath}
\usepackage{amssymb}  
\usepackage{indentfirst}
\usepackage{fancyhdr}
\usepackage{amsthm}
\usepackage{graphicx}
\usepackage{pdfpages}
\usepackage{esint}
\usepackage{url}

\usepackage[numbers,sort&compress]{natbib}

\newcommand{\nc}{\normalcolor}

\newcommand{\red}{\color{black}}

\numberwithin{equation}{section}

\theoremstyle{plain} 
\newtheorem{thm}{Theorem}[section] 
\newtheorem{cor}[thm]{Corollary} 
\newtheorem{lem}[thm]{Lemma} 
\newtheorem{prop}[thm]{Proposition} 
 
\newtheorem{dfn}[thm]{Definition}

\usepackage{xcolor}                    
\definecolor{custom-blue}{RGB}{0,99,166} 
\usepackage{hyperref}
\hypersetup{colorlinks=true, allcolors=custom-blue}

\usepackage{geometry}
\allowdisplaybreaks[4]

\begin{document}

\author{$\text{\sc{Raffaella Giova}}^\blacksquare$$, \text{\sc{Antonio Giuseppe Grimaldi}}^\clubsuit$, $\text{\sc{Stefania Russo}}^\spadesuit$ 
}

\title{  Almost Lipschitz regularity   for solutions of elliptic equations with discontinuous coefficients}

\maketitle

\begin{abstract}
\noindent We are interested in the local  higher  integrability of solutions to elliptic equations with linear growth  of the form 
 $$-\text{ div}A(x,Du)=f(x). $$
\red Under a Besov regularity on the partial map  $x \mapsto A(x,\xi)$ and a summability condition on the datum $f$ in the scale of Lebesgue spaces\nc, we prove that  the solutions are almost Lipschitz continuous, i.e.\ their gradients  belong  locally to $L^q$, for any finite exponent $q$. In turn, solutions are locally $\gamma$-H\"older continuous, for every $\gamma \in (0,1)$.  The difficulty arising from the lack of an explicit second variation for the problem is overcome by testing the equation with a function proportional to a power of the finite difference quotient of the solution. 
%through a Moser-type iteration scheme involving fractional derivatives, which,
To the best of our knowledge, this technique  is used in this context for the first time.  We also provide an example showing the sharpness of our result  in the scale of Lebesgue spaces. 
\end{abstract}


\let\thefootnote\relax\footnotetext{
			\small $^{\blacksquare}$DiSEG, Università degli Studi di Napoli ``Parthenope'', Via Generale Parisi, 13, 80132 Napoli, Italy. E-mail: \textit{raffaella.giova@uniparthenope.it}. https://orcid.org/0000-0002-7632-4664}
\let\thefootnote\relax\footnotetext{
			\small $^{\clubsuit}$Dipartimento di Ingegneria, Università degli Studi di Napoli ``Parthenope'',
Centro Direzionale Isola C4, 80143 Napoli, Italy. E-mail: \textit{antoniogiuseppe.grimaldi@collaboratore.uniparthenope.it}. https://orcid.org/0000-0001-5400-9092}
\let\thefootnote\relax\footnotetext{
			\small $^{\spadesuit}$Dipartimento di Matematica e Applicazioni ``R. Caccioppoli'', Università degli Studi di Napoli ``Federico II'', Via Cintia, 80126 Napoli,
 Italy. E-mail: \textit{stefania.russo3@unina.it}. https://orcid.org/0009-0006-4594-3389}

\section{Introduction}
\noindent  The aim of this paper  is to  prove  the almost Lipschitz continuity of the local weak solutions of equations of the form 
\begin{equation}
    -\text{ div}A(x,Du)=f(x) \qquad \text{in } \Omega, \label{equation}
\end{equation}
where $\Omega \subset \mathbb{R}^n$ is an open set, $n\geq 2$, $A: \Omega \times \mathbb{R}^n \to \mathbb{R}^n$ is a Carath\'eodory function  such that $A(\cdot,0) \in L^2_{\textrm{loc}}(\Omega)$  and $\red f \in L^{\frac{n}{ \alpha}}_{\textrm{loc}}(\Omega) \nc$, for $\alpha \in (0,1)$. We assume that there exist positive constants $\nu, L $ such that

\begin{equation}\label{A1}
    \langle A(x, \xi)-A(x, \eta),\xi-\eta\rangle  \ge \nu|\xi - \eta|^2 \tag{A1}
\end{equation}
\begin{equation}\label{A2}
     |A(x, \xi)-A (x, \eta)| \le L|\xi - \eta| \tag{A2}
\end{equation}

\noindent for a.e.\ $x \in \Omega$ and all $\xi,\eta \in \mathbb{R}^{ n}$.
Concerning the dependence on the $x$-variable, we assume that
there exists a sequence of measurable non-negative functions $g_k \in L^\frac{n}{\alpha}_{\textrm{loc}}(\Omega)$ such that for every $\Omega' \Subset \Omega$
\begin{equation}
    \displaystyle\sum_{k=0}^{\infty} \Vert g_k \Vert^{p}_{L^\frac{n}{\alpha}(\Omega')} < \infty, \label{gk}
\end{equation}
for $1 \le p \le \frac{2n}{n-2\alpha} $, and at the same time  there exists $\mu \in [0,1]$ such that 
\begin{equation}\tag{A3}
    |A(x,\xi)-A(y, \xi)| \leq |x-y|^{\alpha} (g_k(x)+g_k(y))  (\mu^2+|\xi|^2)^\frac{1}{2}, \label{A3}
\end{equation} 
\noindent for a.e.\ $x,y \in \Omega$ such that $2^{-k} \text{diam}(\Omega) \leq |x-y| < 2^{-k+1}\text{diam}(\Omega)$ and for every $\xi \in \mathbb{R}^{ n}$. We note that assumption \eqref{A3} tells us that the map $x \mapsto A(x,\xi)$ belongs  locally  to the Besov space $B^\alpha_{\frac{n}{\alpha},p}(\Omega)$ (see Theorem \ref{def Bes} below). 
%\blue Prof, c'era un punto interrogativo, ma non abbiamo capito a cosa si riferisse : The study of the regularity of solutions to variational problems with Besov coefficients has recently attracted considerable attention (see for examples \cite{Baison,balci,bc,Clop,DahC,Giova,gps,GI,GR,km2}).\nc

\begin{dfn}
A function $u\in W_{\textrm{loc}}^{1,2}(\Omega)$
is a \textit{local weak solution} of equation \eqref{equation}
if and only if, for any test function $\varphi\in \mathcal{C}_{0}^{\infty}(\Omega)$,
the following integral identity holds:
$$
\int_{\Omega}\langle A(x,Du),D\varphi\rangle\,dx\,=\,\int_{\Omega}f\varphi\,dx.
$$
\end{dfn}

\noindent A main topic in the analysis of equations as in \eqref{equation} is to investigate how the regularity of the right-hand side $f$  transfers to  the
solutions under suitable assumptions on the coefficients of the operator $A$.
\\ It is known that for solutions of the $p$-Laplacian type equations or systems with divergence form right-hand side such as
$$\mathrm{div} (|Du|^{p-2}Du)= \mathrm{div} (|F|^{p-2}F) \,, \quad p>1 \,, $$
 Calder\'on-Zygmund estimates of the type
\begin{equation}
    F \in L^q_\mathrm{loc}(\Omega) \quad \Longrightarrow \quad Du \in L^q_\mathrm{loc}(\Omega) \,, \quad q>p \,,\label{czest}
\end{equation}
hold. The scalar case was treated in \cite{Iw}, while the vectorial one in \cite{dbm}. \\
The result in \eqref{czest} is still true for solutions of elliptic equations of the form 
\begin{equation}
    \mathrm{div} A(x,Du)= \mathrm{div} (|F|^{p-2}F) \,, \label{sys}
\end{equation}
 where the vector field $A$ is assumed to satisfy standard $p$-growth and ellipticity conditions and a continuity dependence on the $x$-variable (see \cite{Iw}). In the case of systems of the type \eqref{sys}, Kristensen and Mingione \cite{km} proved that \eqref{czest} holds, provided $q$ is less than a threshold parameter if $n>2$, or for any $q$ if $n=2$.

Extensions of Calder\'on-Zygmund estimates to non-uniformly elliptic  problems 
are addressed in \cite{am}, for systems with $p(x)$-growth conditions, and in \cite{cm}, for double phase type equations. We also refer to \cite{min} for a comprehensive Calder\'on-Zygmund theory in the context of measure data problems. In the setting of Morrey-Campanato spaces we recall the recent paper \cite{Seppecher}.

Iwaniec and Sbordone \cite{Iw2,is} developed the theory of the integrability of the gradient of solutions for linear equations with coefficients of vanishing mean oscillation (VMO), which need not be continuous.
A first non-linear counterpart for equations of 
$p$-Laplacian type can be found in \cite{kin}; see also \cite{kin2} for global estimates. 
\\Since then, the regularity theory for solutions to variational problems with VMO coefficients has attracted considerable attention in the last years. Baison et al.\ \cite{Baison} established  Calder\'on-Zygmund estimates \eqref{czest} for solutions to equations  \eqref{sys} in the linear growth case.
In the homogeneous case, that is when $F=0$,  Cupini et al.\ \cite{cupini} proved that
the $W^{1,n}$-regularity of the coefficients implies that the gradient of solutions belongs to $L^q$, for any finite exponent $q$ (recall that $W^{1,n}$ embeds into the space of VMO functions). In turn, this gives that solutions are $\gamma$-H\"older continuous, for any exponent $\gamma \in (0,1)$. The  Lipschitz continuity cannot be expected  without assuming that the coefficients belong to  a Sobolev space strictly contained in  $W^{1,n}$ (\cite{emm,epdn,ggt,gmp}) 
and the right-hand side $f$ belongs to a function space strictly contained in $L^n$ (\cite{cia1,ku,li,stein})  or  to the Lorentz space $L^{n,\infty}$, if in addition it is radially decreasing (\cite{Russo}).
% OPPURE When the datum $f$ is radially decreasing, its integrability can be weakened; for instance, in \cite{Russo} the local boundedness of the gradient is achieved by assuming that $f$ belongs to the Lorentz space $L^{n,\infty}$. 
%We quote \cite{Russo} where the local boundedness of the gradient is achieved assuming that the datum $f$ is radially decreasing and belongs to the larger space $L^{n,\infty}$. 
 \\ Since the works \cite{gg1,man},  it is known that 
if the right-hand side $f \in L^\infty$ and the operator $A$ is continuous in the $x$-variable, then solutions of \eqref{equation} are locally H\"older continuous with exponent $\gamma$, for any $\gamma \in (0,1)$, but not locally Lipschitz continuous. 
%If in addition $A$ is a H\"older continuous function of the $x$-variable, then the gradient of solutions is H\"older continuous. 
%(see e.g.\ \cite{gg1,man}).

%Our analysis fits into the setting of problems with VMO coefficients; 
Our analysis naturally fits into  the setting of problems with VMO coefficients; indeed,  the Besov  assumption \eqref{A3} implies that the map $x \mapsto A(x,\xi)$ is \textit{locally uniformly in VMO} (see Section \ref{secnot} for the precise definition).

 Regarding variational problems with Besov coefficients, several results concerning the higher fractional differentiability of solutions have been established in recent years (see for example \cite{Baison,balci,bc,Clop,DahC,Giova,gps,GI,GR,km2}). 
%Here, we are interested in the integrability properties of the gradient of solutions $u$ of the equation \eqref{equation}. More precisely, we shall prove that $Du \in L^q_{\mathrm{loc}}(\Omega)$, for every $q >2$, provided we assume $f \in B^\alpha_{\frac{n}{2 %\alpha},p,\textrm{loc}}(\Omega)$.
% We remark that the differentiability assumption on $f$ gives that $f \in L^\frac{n}{\alpha}_{\mathrm{loc}}(\Omega)$, by Lemma \ref{3.1} below.
%Hence,  our result is not covered by previous paper on the Calder\'on-Zygmund theory and it can be seen   as an improvement of the classical Calder\'on-Zygmund theory, in the sense that we show that the gradient of the solution is as integrable as we want by assuming a mild summability condition on the right-hand side $f$,  that is somehow compensated with a small differentiability assumption.
Here,  we are interested in  the higher integrability properties of the gradient for solutions $u$ to equation \eqref{equation}. More precisely, we establish that $Du \in L^q_{\mathrm{loc}}(\Omega)$ for every finite $q > 2$, under the hypothesis 
\red that the right-hand side $f$ belongs to $L^{\frac{n}{\alpha}}_{\mathrm{loc}}(\Omega)$\nc.
%the local Besov space $f \in B^\alpha_{\frac{n}{2 \alpha},p,\mathrm{loc}}(\Omega)$.
% It is worth noting that, by virtue of Lemma \ref{3.1} below, this fractional differentiability guarantees the summability condition $f \in L^{\frac{n}{\alpha}}_{\mathrm{loc}}(\Omega)$.
 Crucially, this setting falls outside the scope of classical Calder\'on-Zygmund estimates. Our result provides a novel extension of the standard Calder\'on-Zygmund theory: we demonstrate that arbitrary higher integrability of the gradient can be attained under a mild summability condition on $f$. 
The classical strategy for establishing Lipschitz regularity of solutions typically relies on differentiating the underlying equation and choosing a test function proportional to a suitable power of the gradient $Du$. However, in our setting, this standard procedure is precluded: under assumption \eqref{A3}, the coefficients possess merely fractional differentiability and lack the smoothness required to justify weak differentiation.
 To overcome this obstruction, we build upon a discrete differentiation approach pioneered by B\"ogelein et al.\ \cite{Boge} for the fractional $p$-Laplacian. The core strategy consists in testing equation \eqref{equation} not with the gradient, but with a carefully chosen function proportional to a power of the finite difference quotient $\tau_h u$. By rigorously coupling this technique with a  
 %tailored fractional
Moser iteration scheme, we establish the \emph{almost Lipschitz} continuity of solutions. Specifically, we prove that $u \in \mathcal{C}^{0,\gamma}_{\mathrm{loc}}(\Omega)$ for every exponent $0 < \gamma <1$.
 It is crucial to emphasize that the lack of $L^\infty_{\mathrm{loc}}$ bounds for the gradient is not a technical  limitation  of our proof, but rather an intrinsic feature of the problem. Because our coefficients belong to the Besov space $B^\alpha_{r,p}$ at the \emph{critical} exponent $r=\frac{n}{\alpha}$, they may fail to be continuous. Consequently, the almost Lipschitz regularity achieved in Theorem \ref{main thm} cannot be improved to Lipschitz continuity. \red Moreover, in the paper \cite{dfm}, the authors proved that requiring the data $f$ to belong $L^{m}_{\mathrm{loc}}$ for $m >n / \alpha$ is necessary to ensure the local boundedness of the gradient of the solution $u$ to elliptic equations with $\mathcal{C}^{0,\alpha}$-regular coefficients, and stated that they do not expect that $Du \in L^\infty_{\mathrm{loc}}$ assuming $f \in L^{\frac{n}{\alpha}}_{\mathrm{loc}}$. \nc We validate this sharpness by constructing a counterexample in Section \ref{sec ex}.
 %, thus closing the question of optimal regularity in this critical fractional regime.

We now state our main result.
\begin{thm}\label{main thm}
  Let $0< \alpha <1$ and $1 \le p \le \frac{2n}{n-2\alpha} $. Let $u \in W^{1,2}_{\textrm{loc}}(\Omega)$ be a local weak solution of \eqref{equation}  and $\red f \in L^{\frac{n}{ \alpha}}_{\textrm{loc}}(\Omega) \nc$. Then, $u \in W^{1,q}_{\textrm{loc}}(\Omega)$ for every $q \in [2,+\infty)$ and the following estimate 
\begin{align}
  \left(  \int_{ B_{R/2} } |Du|^{q} dx \right)^{\frac{1}{q}}  & \leq  \tilde{c}\left( \red \Vert f \Vert_{L^{\frac{n}{\alpha}}(B_{2R})} \nc+   \Vert u \Vert_{W^{1,2}(B_{2R})}  +1 \right) \notag
\end{align}
holds for every concentric balls $B_{R/2} \subset B_{2R} \Subset \Omega $ and for a positive constant  $\tilde{c}=\tilde{c}(n, q,\alpha, \nu, L,p,R)$.  
\end{thm}
\noindent A straightforward consequence of the previous result is the following
\begin{cor}
   Let $0< \alpha <1$ and $1 \le p \le \frac{2n}{n-2\alpha} $. Let $u \in W^{1,2}_{\textrm{loc}}(\Omega)$ be a local weak solution of \eqref{equation} and $\red f \in L^{\frac{n}{ \alpha}}_{\textrm{loc}}(\Omega) \nc$. Then, $u \in \mathcal{C}^{0,\gamma}_{\textrm{loc}}(\Omega)$, for every $0 < \gamma <1$, and we have the following estimate 
\begin{align}
  [u]_{\mathcal{C}^{0,\gamma}(B_{R/2})} & \leq  \tilde{c}\left(\red \Vert f \Vert_{L^{\frac{n}{\alpha}}(B_{2R})} \nc+  \Vert u \Vert_{W^{1,2}(B_{2R})}  +1   \right), \notag
\end{align}
for every concentric balls $B_{R/2} \subset B_{2R} \Subset \Omega $ and for a positive constant  $\tilde{c}=\tilde{c}(n, \gamma,\alpha, \nu, L,p,R)$. 
\end{cor}

\noindent A key role in the proof of Theorem \ref{main thm} is to derive suitable estimates on second-order finite differences of the solution $u$.
The first step is to establish some a priori estimates for sufficiently regular solutions.
Inspired by the work \cite{Boge}, we test the equation \eqref{equation} with 
$$\varphi = \tau_{-h}(\eta^2 |\tau_{h}u|^{q-2} \tau_{h}u),$$
where $q \ge 2 $ and $\eta$ is a cut-off function. Thus, we obtain uniform estimates for the $B^{\frac{2\alpha}{q}+1}_{q,\frac{pq}{2}}$ norm of $ \tau_h(\tau_h u)$, for every $q \geq 2$.
\iffalse
The next step relies on improving the integrability exponent of $Du$ from $2$ to any $q>2$.
The approach is similar in spirit to a Moser-type iteration. If $Du \in L^q_{\textrm{loc}}(\Omega)$ for some $q \ge 2$, the estimate on second-order finite differences, together with Proposition \ref{prop1}, ensures that $Du \in B^\frac{2\alpha}{q}_{q,\frac{pq}{2},\textrm{loc}}(\Omega)$. This implies $Du \in L^{\frac{nq}{n-2\alpha}}_{\textrm{loc}}(\Omega)$  by the Sobolev-type embedding for Besov spaces. 
The scheme is iterated only finitely many times, until the target integrability exponent is achieved. In contrast with Moser’s iteration for the Laplace equation, this method cannot be carried out infinitely many times.

Finally, Theorem \ref{main thm} follows by proving that the a priori estimate is preserved in passing to the limit.
\fi

The next step is to improve the integrability of the gradient $Du$ from the natural energy space $L^2$ to $L^q$, for any arbitrarily large exponent $q>2$.
 To achieve this, we  use  a bootstrap  argument  inspired by Moser's iteration, but  suitable  adapted to the fractional setting. Specifically, assuming $Du \in L^q_{\mathrm{loc}}(\Omega)$ for some $q \ge 2$, our second-order finite difference estimates, combined with Proposition \ref{prop1}, yield the refined Besov regularity $Du \in B^{\frac{2\alpha}{q}}_{q,\frac{pq}{2},\mathrm{loc}}(\Omega)$. By invoking sharp Sobolev-type embeddings for Besov spaces, this fractional differentiability translates directly into the higher integrability $Du \in L^{\frac{nq}{n-2\alpha}}_{\mathrm{loc}}(\Omega)$.
 A key difference from the classical Moser iteration for the Laplace equation is the finite nature of our scheme. Due to the structural properties of the fractional embeddings, the bootstrap cannot be carried out infinitely many times to yield a  $L^\infty$ bound. Nevertheless, this finite sequence of iterations is perfectly calibrated to reach any targeted higher integrability exponent.
Finally, Theorem \ref{main thm} follows by proving that the a priori estimate is preserved in passing to the limit.
\\
We conclude this introduction by describing the organization of the paper. After a brief presentation of the notation and preliminary results,
in Section \ref{secbesov} we state some elementary properties and fundamental estimates for Besov spaces.  
Section \ref{secapriori} is devoted to the proof of Theorem \ref{main thm}. Eventually, in Section \ref{secapp} we prove our main result by using an approximation argument. We conclude exhibiting an example of equation for which the result in Theorem \ref{main thm} is sharp, in the sense that we cannot expect to obtain local Lipschitz continuity of solutions to \eqref{equation} under the assumption \eqref{A3}. This counterexample was constructed by Jin et al.\ \cite{jin} to show that solutions to elliptic equations
with continuous coefficients may not be Lipschitz continuous.

\section{Notation and preliminary results}\label{secnot}

\noindent Here, we introduce some notation.
 We will use the symbols $C$ or $c$ to denote general positive constants. Different occurrences from line to line will be still denoted using the same letters. Relevant dependencies on parameters will be emphasized using parentheses or subscripts. 
We denote by $B(x,r)=B_{r}(x)= \{ y \in \mathbb{R}^{n} : |y-x | < r  \}$ the ball centered at $x$ of radius $r$. We shall omit the dependence on the center and on the radius when no confusion arises. 
%Moreover, we set $S_r := \{ x \in \mathbb{R}^n : |x| = r \}$ the sphere of radius $r$ on %\mathbb{R}^n$. 
For a function $v \in L^{1}(B)$, the symbol
\begin{center}
$v_B:=\displaystyle\fint_{B} v(x) d x = \dfrac{1}{|B|} \displaystyle\int_{B} v(x) d x$.
\end{center}
will denote the integral mean of the function $v$ over the ball $B$.

%\ste che vogliamo fare con la funzione V?\\
%Let $1 \le p < + \infty$. We define an auxiliary function by
%$$V_{p}(\xi):=|\xi|^\frac{p-2}{2} \xi,$$
%for all $\xi\in \mathbb{R}^{n}$. One can easily check that, for $p \geq 2$, it holds
%\begin{gather}
%|\xi|^p \leq  |V_p(\xi)|^2. \label{Vp}
%\end{gather}
%For the auxiliary function $V_{p}$, 
 We recall the following estimate (see e.g. \cite[Lemma 8.3]{giusti}). 
\begin{lem}\label{D1}
Let $1<p<+\infty$. There exists a constant $c=c(n,p)>0$ such that
\begin{center}
$c^{-1}(|\xi|^{2}+|\eta|^{2})^{\frac{p-2}{2}} |\xi-\eta|^{2}\leq {\left ||\xi|^\frac{p-2}{2} \xi-|\eta|^\frac{p-2}{2} \eta \right|^{2}} \leq c(|\xi|^{2}+|\eta|^{2})^{\frac{p-2}{2}} |\xi-\eta|^{2}$
\end{center}
for any $\xi, \eta \in \mathbb{R}^{n} $.
\end{lem}

Now we state a well-known iteration lemma (see \cite[Lemma 6.1]{giusti} for the proof).
\begin{lem}\label{lm2}
Let $\Phi  :  [\frac{R}{2},R] \rightarrow \mathbb{R}$ be a bounded nonnegative function, where $R>0$. Assume that for all $\frac{R}{2} \leq r < \rho \leq R$ it holds
$$\Phi (r) \leq \theta \Phi(\rho) +A + \dfrac{B}{(\rho-r)^2}+ \dfrac{C}{(\rho-r)^{\gamma}},$$
where $\theta \in (0,1)$, $A$, $B$, $C \geq 0$ and $\gamma >0$ are constants. Then there exists a constant $c=c(\theta, \gamma)$ such that
$$\Phi \biggl(\dfrac{R}{2} \biggr) \leq c \biggl( A+ \dfrac{B}{R^2}+ \dfrac{C}{R^{\gamma}}  \biggr).$$
\end{lem}

As usual, we denote by $\mathcal{C}^{0,\gamma}(\Omega)$, with $0< \gamma < 1$, the space of functions $v$ such that the following seminorm $$[v ]_{C^{0,\gamma} (\Omega)} =  \sup_{x,y \in \Omega, x \neq y} \frac{|v(x)-v(y)|}{|x-y|^{\gamma}} $$
is finite.

Given a ball $B_r(x_0) \subset \Omega$, let us define
$$ {A}_{B_r(x_0)}(\xi)= \displaystyle\fint_{B_r(x_0)}{A}(x,\xi)dx .$$ 
One can easily check that if ${A}(x,\xi)$  satisfies  \eqref{A1} and \eqref{A2} they also hold true  for the operator ${A}_{B}(\xi)$. Setting
\begin{equation}\label{dfnH}
    \mathcal{H}(x,B_r(x_0)):= \sup_{\xi \neq 0} \dfrac{|{A}(x,\xi)-  {A}_{B_r(x_0)}(\xi)|}{ (\mu^2 + |\xi|^2)^{\frac 1 2}},
\end{equation}
we will say that the map $x \mapsto {A}(x,\xi)$ is \textit{locally uniformly in VMO} if for each compact set $K \subset \Omega$ we have that
\begin{equation}\label{vmounif}
    \displaystyle\lim_{R \rightarrow 0} \sup_{r <R} \sup_{x_0 \in K} \displaystyle\fint_{B_{r}(x_0)} \mathcal{H}(x,B_r(x_0))dx=0. 
\end{equation}

We have that if  ${A}(x,\xi)$  satisfies \eqref{A3} then it has the locally uniform VMO property (see \cite{Baison} for the proof).
\begin{lem}
    Let $A: \Omega \times \mathbb{R}^n \to \mathbb{R}$ be such that \eqref{A1}, \eqref{A2} and \eqref{A3} hold. Then, $A$ is locally uniformly in VMO.
\end{lem}

\section{Difference quotient}
\label{secquo}
\noindent We recall some properties of the finite difference quotient operator that will be needed in the sequel. 

\begin{dfn}
Let $F$ be a function defined in an open set $\Omega \subset \mathbb{R}^n$ and let $h \in \mathbb{R}^n$. We call the difference quotient of $F$ with respect to $h$ the function
$$ \tau_{h}^1 F(x) =\tau_{h}F(x) :=F(x+h)-F(x) .$$
Moreover, for every $r  \in \mathbb{N} $ we define the $(r+1)$-th finite difference quotient with respect to $h$ as
\begin{align*}
\tau^{r+1}_{h}F(x):= \tau_{h} (\tau^{r}_{h}F(x)).
\end{align*}
\end{dfn}

The function $\tau_{h}F$ is defined in the set
%$$\tau_{h}\Omega := \{  x \in \Omega : x+h \in \Omega \},$$
%and hence in the set
$$\Omega_{|h|}: = \{ x \in \Omega : \mathrm{dist}(x,\partial \Omega)> |h|  \}.$$

We start with the description of some elementary properties that can be found, for example, in \cite{giusti}.
\begin{prop}\label{rapportoincrementale}
Let $F \in W^{1,p}(\Omega)$, with $p \geq1$, and let $G:\Omega \rightarrow \mathbb{R}$ be a measurable function.
Then
\\(i) $\tau_{h}F \in W^{1,p}(\Omega_{|h|})$ and 
$$D_{i}(\tau_{h}F)=\tau_{h}(D_{i}F).$$
(ii) If at least one of the functions $F$ or $G$ has support contained in $\Omega_{|h|}$, then
$$\displaystyle\int_{\Omega}F \tau_h G   dx = \displaystyle\int_{\Omega} G \tau_{-h}F dx.$$
(iii) We have $$\tau_{h} (FG)(x)= F(x+h)\tau_{h} G(x)+G(x) \tau_{h} F(x).$$
\end{prop}
The next result about the finite difference operator is a kind of integral version of Lagrange Theorem.
\begin{lem}\label{ldiff}
If $0<\rho<R,$ $|h|<\frac{R-\rho}{2},$ $1<p<+\infty$ and $F\in W^{1,p}(B_{R})$, then
\begin{center}
$\displaystyle\int_{B_{\rho}} |\tau_{h}F(x)|^{p} dx \leq c(n,p)|h|^{p} \displaystyle\int_{B_{R}} |DF(x)|^{p} dx$.
\end{center}
Moreover, it holds
\begin{center}
$\displaystyle\int_{B_{\rho}} |F(x+h)|^{p} d x \leq  \displaystyle\int_{B_{R}} |F(x)|^{p}d x$.
\end{center}
\end{lem}

\section{Besov spaces}
\label{secbesov}
\noindent We give the definition of Besov spaces as done in \cite[Section 2.5.12, Theorem 1]{Triebel}.
\begin{dfn}
 Let $1 \le p < +\infty$ and let $\alpha >0 $ be a positive real number. Denote by $r$ the smallest integer larger than $\alpha$.
\\Let $1 \leq q< +\infty$.  We say that a function $v : \mathbb{R}^n \rightarrow \mathbb{R}$ belongs to the Besov space $B^{\alpha}_{p,q}(\mathbb{R}^{n})$ if, and only if, $v \in L^{p}(\mathbb{R}^{n})$ and
 \begin{equation}
[v]_{B^{\alpha}_{p,q}(\mathbb{R}^{n})}: =  \biggl( \displaystyle\int_{\mathbb{R}^{n}} \biggl( \displaystyle\int_{\mathbb{R}^{n}} \dfrac{|\tau^r_hv(x)|^{p}}{|h|^{\alpha p}} dx \biggr)^{\frac{q}{p}}  \dfrac{dh}{|h|^{n}} \biggr)^{\frac{1}{q}} < + \infty  . \label{norm1}
\end{equation}
Equivalently, we could simply say that $v \in L^{p}(\mathbb{R}^{n})$ and $\frac{\tau^r_{h}{v}}{|h|^{\alpha}} \in L^{q}\bigl( \frac{dh}{|h|^{n}}; L^{p}(\mathbb{R}^{n}) \bigr)$.
\\When $q = + \infty$, the Besov space $B^{\alpha}_{p,\infty}(\mathbb{R}^{n})$ consists of the functions $v \in L^{p}(\mathbb{R}^{n})$ such that
\begin{equation}
[v]_{B^{\alpha}_{p, \infty}(\mathbb{R}^{n})} :=  \displaystyle\sup_{h \in \mathbb{R}^{n}} \biggl( \displaystyle\int_{\mathbb{R}^{n}} \dfrac{|\tau^r_hv(x)|^{p}}{|h|^{\alpha p}} dx \biggr)^{\frac{1}{p}} < +\infty . \label{norm2}
\end{equation}

\noindent Accordingly, for $1 \le q \le + \infty$, the Besov space $B^{\alpha}_{p,q}(\mathbb{R}^{n})$ is normed with
\begin{equation}
\Vert v \Vert_{B^{\alpha}_{p,q}(\mathbb{R}^{n})}: = \Vert v \Vert_{L^{p}(\mathbb{R}^{n})} + [v]_{B^{\alpha}_{p,q}(\mathbb{R}^{n})} . \label{norm}
\end{equation}
\end{dfn}
Observe that, if $1 \le q < + \infty$, integrating for $h \in B(0, \delta)$ for a fixed $\delta >0,$  an equivalent norm is obtained, because
\begin{center}
$\biggl( \displaystyle\int_{\{|h| \geq \delta\}} \biggl( \displaystyle\int_{\mathbb{R}^{n}} \dfrac{|\tau^r_h v(x)|^{p}}{|h|^{\alpha p}} dx \biggr)^{\frac{q}{p}}  \dfrac{dh}{|h|^{n}} \biggr)^{\frac{1}{q}} \leq c(n, \alpha,p,q, \delta) \Vert v \Vert_{L^{p}(\mathbb{R}^{n})} $.
\end{center}
In the case $q = + \infty$, one can simply take the supremum over $|h| \leq \delta$ and obtain an equivalent norm. By construction, one has $B^{\alpha}_{p, q}(\mathbb{R}^{n}) \subset L^{p}(\mathbb{R}^{n})$. One also has the following version of the Sobolev embeddings (a proof can be found at \cite[Proposition 7.12]{haroske}).
\begin{lem}\label{3.1}
Suppose that $0 < \alpha <1$.
\\ (a) If $1 < p < \frac{n}{\alpha}$ and $1 \leq q \leq p^{*}_{\alpha} = \frac{np}{n- \alpha p}$, then there is a continuous embedding $B^{\alpha}_{p, q}(\mathbb{R}^{n}) \subset L^{p^{*}_{\alpha}}(\mathbb{R}^{n})$.
\\ (b) If $p = \frac{n}{\alpha}$ and $1 \leq q \leq + \infty$, then there is a continuous embedding $B^{\alpha}_{p, q}(\mathbb{R}^{n}) \subset BMO(\mathbb{R}^{n})$,
\\ where $BMO$ denotes the space of functions with bounded mean oscillations \emph{\cite[Chapter 2]{giusti}}.
\end{lem}
We recall the following inclusions between Besov spaces (\cite[Proposition 7.10 and Formula (7.35)]{haroske}).
\begin{lem}\label{3.2}
Suppose that $0 < \beta < \alpha < 1 $.
\\ (a) If $1 < p < +\infty$ and $1 \leq q \leq t \leq + \infty$, then $B^{\alpha}_{p, q}(\mathbb{R}^{n}) \subset B^{\alpha}_{p, t}(\mathbb{R}^{n})$.
\\ (b) If $1 < p < +\infty$ and $1 \leq q , t \leq + \infty$, then $B^{\alpha}_{p, q}(\mathbb{R}^{n}) \subset B^{\beta}_{p, t}(\mathbb{R}^{n})$.
\\ (c) If $1 \le q \le + \infty$, then $B^{\alpha}_{\frac{n}{\alpha}, q}(\mathbb{R}^{n}) \subset B^{\beta}_{\frac{n}{\beta}, q}(\mathbb{R}^{n})$.
\end{lem}

Differences in \eqref{norm1} and \eqref{norm2} can be replaced by derivatives (see \cite[Remark 2, page 113]{Triebel}). Let $1\leq p,q<+\infty$ and
$$\alpha=\sigma+m, \quad 0 < \sigma < 1 \quad \text{and} \quad m \in \mathbb{N}_0. $$
Then, the norm
\begin{equation}
    \sum_{|k|\le m} \Vert D^k v \Vert_{L^p(\mathbb{R}^n)} + \sum_{|k|\le m} \biggl( \displaystyle\int_{\mathbb{R}^{n}} \biggl( \displaystyle\int_{\mathbb{R}^{n}} \dfrac{|\tau_hD^k v(x)|^{p}}{|h|^{\sigma p}} dx \biggr)^{\frac{q}{p}}  \dfrac{dh}{|h|^{n}} \biggr)^{\frac{1}{q}} \label{equivnorm}
\end{equation}
is equivalent to \eqref{norm} (with the usual modification if $q=+\infty$).

Given a domain $\Omega \subset \mathbb{R}^{n}$, we say that a function $v: \mathbb{R}^n \rightarrow \mathbb{R}$ belongs to the local Besov space $ B^{\alpha}_{p, q,\text{loc}}$ if $\varphi  v \in B^{\alpha}_{p, q}(\mathbb{R}^{n})$ whenever $\varphi \in \mathcal{C}^{\infty}_{0}(\Omega)$. It is worth noticing that one can prove suitable version of Lemmas \ref{3.1} and \ref{3.2}, by using local Besov spaces.

 The following lemma provides a characterization of local Besov spaces $B^{\alpha}_{p,q,\text{loc}}$, which will be fundamental for our subsequent analysis (for the proof see \cite[Lemma 7]{Baison}, in the case $0< \alpha <1$, and \cite[Lemma 2.9]{Grimaldi}, in the case $1 \le \alpha <2$). 
\begin{lem}
Let $1 \le p < + \infty$, $1 \le q \le + \infty$ and $0< \alpha < 2$.
A function $v \in L^{p}_{\text{loc}}(\Omega)$ belongs to the local Besov space $B^{\alpha}_{p,q,\text{loc}}$, if, and only if,
\begin{center}
$\biggl\Vert \dfrac{\tau^r_{h}v}{|h|^{\alpha}} \biggr\Vert_{L^{q}\bigl(\frac{dh}{|h|^{n}};L^{p}(B)\bigr)}< + \infty,$
\end{center}
for any ball $B\subset2B\subset\Omega$ with radius $r_{B}$. Here the measure $\frac{dh}{|h|^n}$ is restricted to the ball $B(0,r_B)$ on the h-space.
\end{lem}

Now, we prove the following local version of the Sobolev embedding stated at Lemma \ref{3.1} .
\begin{prop}\label{lemma embedding}
Suppose that $0 < \alpha < 1$, $1 < p < \frac{n}{\alpha}$ and $1 \le q \le p^*_\alpha$. Then, for every function $v \in B^\alpha_{p,q,\text{loc}}(\Omega)$ the following local estimate
\begin{align}
    \Vert v \Vert_{L^\frac{np}{n-\alpha p}(B_\rho)} &\le  c\Vert  v \Vert_{L^p \left(B_\frac{R+\rho}{2} \right)}+ \dfrac{c}{(R-\rho)^\alpha} \Vert v \Vert_{L^p \left(  B_\frac{3R+\rho}{4}   \right)}\notag\\
    & \qquad \qquad + c\biggl( \displaystyle\int_{ |h| \le \frac{R-\rho}{4} } \biggl( \displaystyle\int_{  B_\frac{3R+\rho}{4}   } \dfrac{|\tau_hv(x)|^{p}}{|h|^{\alpha p}} dx \biggr)^{\frac{q}{p}}  \dfrac{dh}{|h|^{n}} \biggr)^{\frac{1}{q}} \label{embedding}
\end{align}
holds for every balls $B_\rho \subset B_{R} \Subset \Omega$ and for a positive constant $c=c(n,p,q, \alpha)$.
\end{prop}
\proof
Let us consider radii $0 < \rho < R $ and  balls $B_\rho \subset B_{R} \Subset \Omega$ and a cut-off function $\eta \in \mathcal{C}_0^\infty(B_\frac{R+\rho}{2})$ such that $0 \le \eta \le 1$, $\eta =1$ on $B_\rho$, $|D\eta| \le \frac{C}{R-\rho}$. By virtue of Lemma \ref{3.1} $(a)$, we have
\begin{equation}
    \Vert v \Vert_{L^\frac{np}{n-\alpha p}(B_\rho)} \le  \Vert \eta v \Vert_{L^\frac{np}{n-\alpha p}(\mathbb{R}^n)} \le c \Vert \eta v \Vert_{B^\alpha_{p,q}(\mathbb{R}^n)},
    \label{stima 1}
\end{equation}
where $c$ is a positive constant depending on  $n$,  $p$, $q$ and $\alpha$.

We take care of the norm $\Vert \eta v \Vert_{B^\alpha_{p,q}(\mathbb{R}^n)}$ in \eqref{stima 1}. It holds
\begin{align}
   \Vert \eta v \Vert_{B^\alpha_{p,q}(\mathbb{R}^n)} &= \Vert \eta v \Vert_{L^p(\mathbb{R}^n)}+[ \eta v ]_{B^\alpha_{p,q}(\mathbb{R}^n)} \notag\\
   & \le \Vert  v \Vert_{L^p \left(B_\frac{R+\rho}{2} \right)} +   \biggl( \displaystyle\int_{ |h| \le \frac{R-\rho}{4} } \biggl( \displaystyle\int_{\mathbb{R}^{n}} \dfrac{\eta(x+h)^p|\tau_hv(x)|^{p}}{|h|^{\alpha p}} dx \biggr)^{\frac{q}{p}}  \dfrac{dh}{|h|^{n}} \biggr)^{\frac{1}{q}} \notag \\
   & \qquad \qquad  +   \biggl( \displaystyle\int_{  |h| > \frac{R-\rho}{4} } \biggl( \displaystyle\int_{\mathbb{R}^{n}} \dfrac{\eta(x+h)^p|\tau_hv(x)|^{p}}{|h|^{\alpha p}} dx \biggr)^{\frac{q}{p}}  \dfrac{dh}{|h|^{n}} \biggr)^{\frac{1}{q}} \notag\\
   & \qquad \qquad  +   \biggl( \displaystyle\int_{ |h| \le \frac{R-\rho}{4} } \biggl( \displaystyle\int_{\mathbb{R}^{n}} \dfrac{|\tau_h\eta|^p|v(x)|^{p}}{|h|^{\alpha p}} dx \biggr)^{\frac{q}{p}}  \dfrac{dh}{|h|^{n}} \biggr)^{\frac{1}{q}} \notag\\
   & \qquad \qquad  +   \biggl( \displaystyle\int_{ |h| > \frac{R-\rho}{4} } \biggl( \displaystyle\int_{\mathbb{R}^{n}} \dfrac{|\tau_h\eta|^p|v(x)|^{p}}{|h|^{\alpha p}} dx \biggr)^{\frac{q}{p}}  \dfrac{dh}{|h|^{n}} \biggr)^{\frac{1}{q}}\notag\\
   & \le \Vert  v \Vert_{L^p    \left(B_\frac{R+\rho}{2} \right)} +A_1+A_2+A_3+A_4, \label{stima2}
\end{align}
 where we used $(iii)$ of Proposition \ref{rapportoincrementale}. 
\\Using that  $\eta \in \mathcal{C}_0^\infty(B_\frac{R+\rho}{2})$ is such that $0 \le \eta \le 1$  and the fact that $ \alpha >0 $, we get the following estimates
\begin{align}
    A_1 \le \biggl( \displaystyle\int_{ |h| \le \frac{R-\rho}{4} } \biggl( \displaystyle\int_{ B_\frac{3R+\rho}{4} } \dfrac{|\tau_hv(x)|^{p}}{|h|^{\alpha p}} dx \biggr)^{\frac{q}{p}}  \dfrac{dh}{|h|^{n}} \biggr)^{\frac{1}{q}}, \label{A}
\end{align}
\begin{align}
    A_2 \le 2 \Vert v \Vert_{L^p \left( B_\frac{3R+\rho}{4}  \right)} \biggl( \underbrace{\displaystyle\int_{  |h| > \frac{R-\rho}{4} } \dfrac{dh}{|h|^{n+\alpha q}}}_{I} \biggr)^{\frac{1}{q}}= \dfrac{c(n,q,\alpha)}{(R-\rho)^\alpha}\Vert v \Vert_{L^p \left( B_\frac{3R+\rho}{4}  \right)}, \label{B}
\end{align}
\begin{align}
    A_4 &  \le \biggl( \displaystyle\int_{|h| > \frac{R-\rho}{4}} \biggl( \displaystyle\int_{B_\frac{3R+\rho}{4}} \dfrac{(\eta(x+h)+\eta(x))^p|v(x)|^{p}}{|h|^{\alpha p}} dx \biggr)^{\frac{q}{p}}  \dfrac{dh}{|h|^{n}} \biggr)^{\frac{1}{q}} \notag\\
    & \le \biggl( \displaystyle\int_{ |h| > \frac{R-\rho}{4} } \biggl( \displaystyle\int_{ B_\frac{3R+\rho}{4}} \dfrac{2^p|v(x)|^{p}}{|h|^{\alpha p}} dx \biggr)^{\frac{q}{p}}  \dfrac{dh}{|h|^{n}} \biggr)^{\frac{1}{q}} \notag\\
    & \le 2  \Vert v \Vert_{L^p \left(  B_\frac{3R+\rho}{4}  \right)}    \biggl(\underbrace{ \displaystyle\int_{ |h| > \frac{R-\rho}{4} }   \dfrac{dh}{|h|^{n+\alpha q}}}_{II} \biggr)^{\frac{1}{q}}  = \dfrac{c(n,q,\alpha)}{(R-\rho)^\alpha} \Vert v \Vert_{L^p \left(  B_\frac{3R+\rho}{4}  \right)}. \label{D}
\end{align}
 The integral $A_3$ can be estimated using the property $|D\eta| \le \frac{C}{R-\rho}$ and the fact that $\alpha <1$ in the following way
\begin{align}
    A_3 &\le   \biggl( \displaystyle\int_{ |h| \le \frac{R-\rho}{4} } \biggl( \displaystyle\int_{\mathbb{R}^{n}}  \Vert D \eta \Vert_{\infty}^p |v(x)|^{p}|h|^{(1-\alpha) p} dx \biggr)^{\frac{q}{p}}  \dfrac{dh}{|h|^{n}} \biggr)^{\frac{1}{q}} \notag\\
    & \le \dfrac{C}{R-\rho}  \Vert v \Vert_{L^p \left(  B_\frac{3R+\rho}{4}   \right)} \biggl( \underbrace{\displaystyle\int_{ |h| \le \frac{R-\rho}{4} } |h|^{(1-\alpha) q}   \dfrac{dh}{|h|^{n}}}_{III}\biggr)^{\frac{1}{q}}  = \dfrac{c(n,q,\alpha)}{(R-\rho)^\alpha} \Vert v \Vert_{L^p \left(  B_\frac{3R+\rho}{4}   \right)}.\label{C}
\end{align}
 Here, the integrals $I$, $II$ and $III$ were computed using polar coordinates. 
\\Now, combining estimates \eqref{stima 1}, \eqref{stima2}, \eqref{A}, \eqref{B}, \eqref{D} and \eqref{C}, we infer
\begin{align}
    \Vert v \Vert_{L^\frac{np}{n-\alpha p}(B_\rho)} &\le c \Vert  v \Vert_{L^p \left(B_\frac{R+\rho}{2} \right)}+ \dfrac{c}{(R-\rho)^\alpha} \Vert v \Vert_{L^p \left(  B_\frac{3R+\rho}{4}   \right)}\notag\\
    & \qquad \qquad + c\biggl( \displaystyle\int_{ |h| \le \frac{R-\rho}{4} } \biggl( \displaystyle\int_{  B_\frac{3R+\rho}{4} } \dfrac{|\tau_hv(x)|^{p}}{|h|^{\alpha p}} dx \biggr)^{\frac{q}{p}}  \dfrac{dh}{|h|^{n}} \biggr)^{\frac{1}{q}}, \notag
\end{align}
for a positive constant $c=c(n,p,q, \alpha)$.
\endproof
\noindent The next result will help us to deal with second order differences in the fractional context.

\begin{prop}\label{prop1}
Let $1 \le p,q < +\infty$, $0 < \alpha < 1$, $M>0$ and $0 <r< \rho < R $. There exists a positive constant $c=c(n,p,q,\alpha)$ such that whenever $v \in L^p(B_R)$ satisfies
\begin{equation}
    \int_{ B_\frac{R-\rho}{8}}\left(  \int_{B_R}\dfrac{|\tau_h(\tau_h v)|^p}{|h|^{p(1+\alpha)}} dx\right)^\frac{q}{p} \dfrac{dh}{|h|^n} \le M^q, \label{ipotesi1}
\end{equation}
then $v \in B^{1+\alpha}_{p,q,\textrm{loc}}(B_R)$ and 
\begin{align}
 \int_{B_{\rho-r}}\left(  \int_{B_r}\dfrac{|\tau_hD v|^p}{|h|^{p\alpha}} dx\right)^\frac{q}{p} \dfrac{dh}{|h|^n} \le &\, c M^q+  c \Vert v \Vert^q_{L^p(B_R)}
    +\frac{c}{(R-\rho)^{q(1+\alpha)}} \Vert v \Vert_{L^p(B_R)}^q \notag\\
   & \qquad +  \frac{c}{(R-\rho)^{q\alpha}} \Vert Dv \Vert_{L^p(B_R)}^q  . \label{stima3}
\end{align}

\end{prop}
\proof  For $ r<\rho $ given in the assumptions, let us consider a cut-off function $\eta \in \mathcal{C}_0^\infty(B_\frac{R+\rho}{2})$ such that $0 \le \eta \le 1$, $\eta =1$ on $B_\rho$, $|D\eta| \le \frac{C}{R-\rho}$ and $|D^2\eta| \le \frac{C}{(R-\rho)^2}$.

The equivalence between the norms \eqref{norm} and \eqref{equivnorm} and the properties of $\eta$ give that
\begin{align}
& \int_{B_{\rho-r}}\left(  \int_{B_r}\dfrac{|\tau_hD v|^p}{|h|^{p\alpha}} dx\right)^\frac{q}{p} \dfrac{dh}{|h|^n} \notag \\
    & \qquad \le \int_{\mathbb{R}^n}\left(  \int_{\mathbb{R}^n}\dfrac{|\tau_hD(\eta v)|^p}{|h|^{p\alpha}} dx\right)^\frac{q}{p} \dfrac{dh}{|h|^n} \notag\\
    & \qquad \le c \Vert v \Vert^q_{L^p(B_R)}+c \int_{ B_\frac{R-\rho}{8}}\left(  \int_{\mathbb{R}^n}\dfrac{|\tau_h(\tau_h(\eta v))|^p}{|h|^{p(1+\alpha)}} dx\right)^\frac{q}{p} \dfrac{dh}{|h|^n} \notag\\
    & \qquad \qquad + c \int_{ |h| >\frac{R-\rho}{8} }\left(  \int_{\mathbb{R}^n}\dfrac{|\tau_h(\tau_h(\eta v))|^p}{|h|^{p(1+\alpha)}} dx\right)^\frac{q}{p} \dfrac{dh}{|h|^n}.\label{stima1!}
\end{align}
The last integral on the right-hand side of \eqref{stima1!}  can be estimated as follows  
\begin{align}
    &\int_{ |h| >\frac{R-\rho}{8}}\left(  \int_{\mathbb{R}^n}\dfrac{|\tau_h(\tau_h(\eta v))|^p}{|h|^{p(1+\alpha)}} dx\right)^\frac{q}{p} \dfrac{dh}{|h|^n} \notag\\
    & \qquad \le  4^q \Vert v \Vert^q_{L^p(B_R)} \displaystyle\int_{  |h| > \frac{R-\rho}{8} } \dfrac{dh}{|h|^{n+(1+\alpha) q}} = \dfrac{c(n,q,\alpha)}{(R-\rho)^{(1+\alpha)q}} \Vert v \Vert_{L^p(B_R)}^q. \label{IIint}
\end{align}
 We observe that, by (iii) of Proposition \ref{rapportoincrementale}, we have that
\begin{equation}
    \tau_h(\tau_h(\eta v))(x)=\eta(x+2h)\tau^2_hv(x)+2 \tau_h\eta(x+h)\tau_hv(x)+v(x)\tau^2_h\eta(x). \label{tau2}
\end{equation}
Therefore, by \eqref{tau2}, the properties of $\eta$ and \eqref{ipotesi1},  the second term on the right-hand side of \eqref{stima1!} can be estimated as follows
\begin{align}
   & \int_{ B_\frac{R-\rho}{8}}\left(  \int_{\mathbb{R}^n}\dfrac{|\tau_h(\tau_h(\eta v))|^p}{|h|^{p(1+\alpha)}} dx\right)^\frac{q}{p} \dfrac{dh}{|h|^n} \notag\\ 
   & \qquad \le 
   c \int_{ B_\frac{R-\rho}{8}}\left(  \int_{\mathbb{R}^n} |\eta(x+2h)|^p \frac{|\tau^2_hv|^p}{|h|^{p(1+\alpha)}} dx \right)^\frac{q}{p} \dfrac{dh}{|h|^n} \notag\\
   & \qquad \qquad + c \int_{ B_\frac{R-\rho}{8}}\left( \int_{\mathbb{R}^n}
   |\tau_h \eta(x+h)|^{p}
    \frac{|\tau_h v |^{p}}{|h|^{p(1+\alpha)}} dx \right)^\frac{q}{p} \dfrac{dh}{|h|^n}\notag\\ 
    & \qquad \qquad + c \int_{ B_\frac{R-\rho}{8} }\left(  
    \int_{\mathbb{R}^n} |v|^p \frac{|\tau^2_h\eta(x)|^p}{|h|^{p(1+\alpha)}} dx \right)^\frac{q}{p} \dfrac{dh}{|h|^n} \notag\\
    & \qquad \le 
   c \int_{ B_\frac{R-\rho}{8}}\left(   \int_{B_\frac{3R+\rho}{4}}  \frac{|\tau^2_hv|^p}{|h|^{p(1+\alpha)}} dx \right)^\frac{q}{p} \dfrac{dh}{|h|^n}\notag\\
   & \qquad \qquad + c \int_{ B_\frac{R-\rho}{8} }\left(   \int_{B_\frac{3R+\rho}{4}}
   \Vert D \eta \Vert^p_{\infty}
     \frac{|\tau_h v |^{p}}{|h|^{p\alpha}} dx \right)^\frac{q}{p} \dfrac{dh}{|h|^n}\notag\\
    & \qquad \qquad 
    + c \int_{ B_\frac{R-\rho}{8}}\left(   
    \int_{B_\frac{3R+\rho}{4}} |h|^{p(1-\alpha)} |v|^p \Vert D^2 \eta\Vert^p_{\infty}  dx \right)^\frac{q}{p} \dfrac{dh}{|h|^n}\notag\\
   & \qquad \le c  M^q +  cI+cII. \label{stima2!}
    \end{align}
    Exploiting Lemma \ref{ldiff} and the fact  $|D \eta| \le \frac{c}{R-\rho}$, we estimate the integral $I$ as follows
\begin{align}
    I \le & \,c\int_{ B_\frac{R-\rho}{8}}\left(   \int_{B_\frac{3R+\rho}{4}}
   \Vert D \eta \Vert^p_{\infty}
     |Dv|^p |h|^{p(1-\alpha)} dx \right)^\frac{q}{p} \dfrac{dh}{|h|^n}\notag\\
     \le & \,\dfrac{c}{(R-\rho)^{q}}\Vert Dv \Vert_{L^p(B_R)}^q \int_{ B_\frac{R-\rho}{8}}|h|^{q(1-\alpha)}  \dfrac{dh}{|h|^n} = \dfrac{c}{(R-\rho)^{q\alpha}}\Vert Dv \Vert_{L^p(B_R)}^q , \label{I int}
\end{align}
where the last integral was computed using polar coordinates.\\
 By using the fact  $|D^2 \eta| \le \frac{c}{(R-\rho)^2}$, we get the following estimate for $II$
\begin{align}
    II \le & \, \dfrac{c}{(R-\rho)^{2q}} \Vert v \Vert_{L^p(B_R)}^q  \int_{ B_\frac{R-\rho}{8}} |h|^{q(1-\alpha)}\dfrac{dh}{|h|^n} = \dfrac{c}{(R-\rho)^{q(1+\alpha)}} \Vert v \Vert_{L^p(B_R)}^q. \label{II int}
\end{align}

Putting \eqref{I int} and \eqref{II int} in \eqref{stima2!}, we find  
\begin{align}
    & \int_{ B_\frac{R-\rho}{8}}\left(  \int_{\mathbb{R}^n}\dfrac{|\tau_h(\tau_h(\eta v))|^p}{|h|^{p(1+\alpha)}} dx\right)^\frac{q}{p} \dfrac{dh}{|h|^n} \notag\\
     & \qquad \le c  M^q +  \dfrac{c}{(R-\rho)^{q\alpha}}\Vert Dv \Vert_{L^p(B_R)}^q  +\dfrac{c}{(R-\rho)^{q(1+\alpha)}} \Vert v \Vert_{L^p(B_R)}^q. \label{stima2!!}
\end{align} 
 Eventually, inserting \eqref{IIint}, \eqref{stima2!!} in \eqref{stima1!}, we obtain the desired estimate \eqref{stima3}. 

\endproof

\noindent Combining Propositions \ref{lemma embedding} and \ref{prop1}, we obtain the following

\begin{cor}\label{corollary}
   Suppose that $0 < \alpha < 1$, $1 < p < \frac{n}{\alpha}$ and $1 \le q \le p^*_\alpha$. Let $M>0$ and $0 <s< t < \overline{R} $. There exists a positive constant $c=c(n,p,q,\alpha)$ such that whenever $u \in L^p( B_{\overline{R}})$ satisfies
\begin{equation}
    \int_{B_\frac{\overline{R}-t}{8}}\left(  \int_{B_{\overline{R}}}\dfrac{|\tau_h(\tau_h u)|^p}{|h|^{p(1+\alpha)}} dx\right)^\frac{q}{p} \dfrac{dh}{|h|^n} \le M^q, \notag
\end{equation}
then $Du \in L^\frac{np}{n-\alpha p}_{\textrm{loc}}(B_{\overline{R}})$ and 
\begin{align}
     \Vert Du \Vert_{L^\frac{np}{n-\alpha p}(B_{s})} &\le  c M+  c \Vert u \Vert_{L^p(B_{ \overline{R}})}
    +\frac{c}{(\overline{R}-t)^{1+\alpha}} \Vert u \Vert_{L^p(B_{ \overline{R}})} +\frac{c}{(\overline{R}-t)^{\alpha}} \Vert Du \Vert_{L^p(B_{ \overline{R}})} \notag\\
    & \qquad + c\Vert  Du \Vert_{L^p(B_{\overline{R}})}+ \dfrac{c}{(t-s)^\alpha} \Vert Du \Vert_{L^p(B_{\overline{R}})}.\notag
\end{align}
\end{cor}
\proof Let us consider radii $0 < s<t < \overline{R} $ and balls $B_{s} \subset B_{t} \subset B_{\overline{R}}  $.
We can use Proposition \ref{prop1} with $R= \overline{R}$, $\rho=t$ and $r= \frac{3t+s}{4}$, thus getting by \eqref{stima3} the following inequality
\begin{align}
\int_{B_\frac{t-s}{4}}\left(  \int_{B_\frac{3t+s}{4}}\dfrac{|\tau_hD u|^p}{|h|^{p\alpha}} dx\right)^\frac{q}{p} \dfrac{dh}{|h|^n} & \le 
\int_{B_\frac{t-s}{4}}\left(  \int_{B_{t}}\dfrac{|\tau_hD u|^p}{|h|^{p\alpha}} dx\right)^\frac{q}{p} \dfrac{dh}{|h|^n} \notag\\
& \le c M^q+  c \Vert u \Vert^q_{L^p(B_{\overline{R}})}
    +\frac{c}{(\overline{R}-t)^{q(1+\alpha)}} \Vert u \Vert_{L^p(B_{\overline{R}})}^q \notag\\
    & \qquad   +\frac{c}{(\overline{R}-t)^{q\alpha}} \Vert Du \Vert_{L^p(B_{\overline{R}})}^q .  \notag
\end{align}
On the other hand, estimate \eqref{embedding} with $v=Du$, $R=t$ and $\rho=s$ reads as
\begin{align}
    \Vert Du \Vert_{L^\frac{np}{n-\alpha p}(B_{s})} &\le  c\Vert  Du \Vert_{L^p \left(B_\frac{t+s}{2} \right)}+ \dfrac{c}{(t-s)^\alpha} \Vert Du \Vert_{L^p \left(B_\frac{3t+s}{4} \right)}\notag\\
    & \qquad \qquad + c\left( \displaystyle\int_{B_{\frac{t-s}{4}}} \left( \displaystyle\int_{B_\frac{3t+s}{4} }\dfrac{|\tau_hDu(x)|^{p}}{|h|^{\alpha p}} dx \right)^{\frac{q}{p}}  \dfrac{dh}{|h|^{n}} \right)^{\frac{1}{q}}. \notag
\end{align}
The conclusion follows combining these two inequalities.
\endproof

 We conclude recalling  that Besov spaces of fractional order $\alpha \in (0,1)$ can be characterized in pointwise terms. We give the following definition according to \cite{koskela}.
\begin{dfn}
 Given a measurable function $v:\mathbb{R}^{n} \rightarrow \mathbb{R}$, a \textit{fractional $\alpha$-Hajlasz gradient for $v$} is a sequence $\{g_{k}\}_{k}$ of measurable, non-negative functions $g_{k}:\mathbb{R}^{n} \rightarrow \mathbb{R}$, together with a null set $N\subset\mathbb{R}^{n}$, such that the inequality 
\begin{center}
$|v(x)-v(y)|\leq (g_{k}(x)+g_{k}(y))|x-y|^{\alpha}$
\end{center} 
holds whenever $k \in \mathbb{Z}$ and $x,y \in \mathbb{R}^{n}\setminus N$ are such that $2^{-k} \leq|x-y|<2^{-k+1}$. We say that $\{g_{k}\}_{k} \in l^{q}(\mathbb{Z};L^{p}(\mathbb{R}^{n}))$ if
\begin{center}
$\Vert \{g_{k}\}_{k} \Vert_{l^{q}(L^{p})}=\biggl(\displaystyle\sum_{k \in \mathbb{Z}}\Vert g_{k} \Vert^{q}_{L^{p}(\mathbb{R}^{n})} \biggr)^{\frac{1}{q}}<  +\infty.$
\end{center} 
\end{dfn}
The following result was proved in \cite{koskela}.
\begin{thm}\label{def Bes}
Let $0< \alpha <1,$ $1 \leq p < +\infty$ and $1\leq q \leq +\infty $. Let $v \in L^{p}(\mathbb{R}^{n})$. Then $v \in B^{\alpha}_{p,q}(\mathbb{R}^{n})$ if, and only if, there exists a fractional $\alpha$-Hajlasz gradient $\{g_{k}\}_{k} \in l^{q}(\mathbb{Z};L^{p}(\mathbb{R}^{n}))$ for $v$. Moreover,
\begin{center}
$\Vert v \Vert_{B^{\alpha}_{p,q}(\mathbb{R}^{n})}\simeq \inf \Vert \{g_{k}\}_{k} \Vert_{l^{q}(L^{p})},$
\end{center}
where the infimum runs over all possible fractional $\alpha$-Hajlasz gradients for $v$.
\end{thm}

\section{A priori estimate}\label{secapriori}

%  We recall that by assumption  $f \in B^\alpha_{\frac{n}{2 \alpha},p,\textrm{loc}}(\Omega)$,  by the very definition of Besov space, we know that   $f \in L^\frac{n}{2\alpha}_{\textrm{loc}}(\Omega)$.   
 
 \noindent By standard elliptic regularity (see for example \cite[Chapter 10]{DBG}), under assumptions \eqref{A1} and \eqref{A2} and $\red f \in L^\frac{n}{\alpha}_{\textrm{loc}}(\Omega)\nc$, local weak solutions $u$ of \eqref{equation} are locally bounded in $\Omega$ and satisfy the following $L^\infty$-bound
\begin{equation}
  \red  \Vert u \Vert_{L^\infty(B_\rho(x_0))} \le \, cR^\frac{2-\alpha}{2} \Vert f \Vert_{L^\frac{n}{\alpha}(B_R(x_0))}+\dfrac{c R^\frac{2n}{2-\alpha}}{(R-\rho)^{2n/(2-\alpha)}} \left( \fint_{B_R(x_0)}u^2 dx  \right)^\frac{1}{2} \nc \label{Linfty bound}
\end{equation}
for every radii $0 < \rho < R$, where $B_R(x_0) \Subset \Omega$ and $c$ is a positive constant depending on $n$, $\alpha$, $\nu$ and $L$.

The next Theorem is devoted to the proof of an a priori estimate, which is the main tool in the proof of Theorem~\ref{main thm}.

\begin{thm}\label{apriori thm}
   Let $\delta \ge 1$, $0< \alpha <1$ and $1 \le p \le \frac{2n}{n-2\alpha} $. Let $u \in  W^{1,2}_{\textrm{loc}}(\Omega) $ be a local weak solution of \eqref{equation} such that $ Du \in L^\frac{n(\delta+1)}{n-2\alpha}_{\textrm{loc}}(\Omega)$. Then, for every $x_0 \in \Omega$ there exists a radius $R_\delta=R_\delta(x_0,n,\delta,\alpha,\nu,L,p)\leq 1$ and a ball $ B_{R_\delta}(x_0) \Subset \Omega $ such that we have the following estimate
   \begin{align}
  \left(  \int_{ B_{r}(x_0) } |Du|^{\frac{n(\delta+1)}{n-2\alpha}} dx \right)^{\frac{n-2\alpha}{n}}  & \leq \frac{\tilde{c}}{(R-r)^{ \red \beta \nc}} \Biggl[  \int_{B_{R}(x_0) }  |u|^{\delta+1} dx  \notag\\
 & \qquad +\int_{  B_{R}(x_0) } |Du|^{\delta+1} dx 
       +  \red \Vert f \Vert^{\delta + 1}_{L^{\frac{n}{\alpha}}( B_{R}(x_0) )} \nc +  1 \Biggr], \label{stima a priori}
\end{align}
for every radii $0<r<R\le R_\delta$, \red for $\beta= \beta(\alpha,\delta)=\max \{ \delta+1, 2 (1+\alpha)  \}$ \nc and for a positive constant  $\tilde{c}=\tilde{c}(n, \delta,\alpha, \nu, L,p)$.
\end{thm}

\begin{proof}
Fix a ball $B_{R_\delta}(x_0) \Subset \Omega$ and consider radii $0<r<\sigma< t'<R \leq R_\delta$  (the radius $R_\delta$ will be choosen later). 
We divide the interval $(\sigma, t')$ into five equal parts by means of $\rho'$, $\rho$, $s$ and $t$, so that $r<\sigma< \rho'< \rho< s<t<t'<R \leq R_\delta$, with 
\begin{equation}
    \rho'-\sigma=\rho-\rho'=s-\rho=t-s=t'-t=\frac{t'-\sigma}{5}. \label{radii}
\end{equation}

For the sake of notation, from now on we shall omit the dependence of the balls on the center $x_0$. 

Let us consider a  cut-off function $\eta \in C_0^{\infty}(B_t)$, with $\eta=1$ on $B_{s}$, $0 \leq \eta \leq 1$, $|D \eta | \leq \frac{c}{t-s}$  and $ |h|\leq \frac{t' - \sigma}{40}$.
We test the equation \eqref{equation} with the function $$\varphi = \tau_{-h}(\eta^2 |\tau_{h}u|^{\delta-1} \tau_{h}u) \qquad (\delta \ge 1)$$ 
and  then we use (ii) of Proposition \ref{rapportoincrementale}, thus getting
    \begin{align}
        I= \int_\Omega \langle \, \tau_h \left(A(x, Du) \right), D(\eta^2 |\tau_h u|^{\delta -1} \tau_h u) \, \rangle dx = \red \int_\Omega f \, \tau_{-h}(\eta^2  |\tau_h u|^{\delta-1} \tau_h u ) dx \nc = I_5. \label{equality1}
    \end{align}
The integral $I$ can be written as follows   
    \begin{align}
        I&= \int_\Omega \langle A(x+h, Du(x+h))- A(x+h,Du(x)) , D(\eta^2 |\tau_h u|^{\delta-1} \tau_h u) \rangle dx \notag \\
        & \qquad + \int_\Omega \langle A(x+h, Du(x))- A(x,Du(x)) , D(\eta^2 |\tau_h u|^{\delta-1} \tau_h u) \rangle dx \notag \\
         &= \delta  \int_\Omega \langle A(x+h, Du(x+h))- A(x+h,Du(x)) , \eta^2 |\tau_h u|^{\delta-1} \tau_h Du \,  \rangle dx \notag \\
        & \qquad + 2 \int_\Omega \langle A(x+h, Du(x+h))- A(x+h,Du(x)) , \eta D\eta|\tau_h u|^{\delta-1} \tau_h u \, \rangle dx \notag \\
         & \qquad + \delta  \int_\Omega \langle A(x+h, Du(x))- A(x,Du(x)) , \eta^2 |\tau_h u|^{\delta-1} \tau_h Du \, \rangle dx \notag \\
         & \qquad + 2 \int_\Omega \langle A(x+h, Du(x))- A(x,Du(x)) , \eta D\eta|\tau_h u|^{\delta-1} \tau_h u \, \rangle dx \notag \\
         &= I_1 +I_2 + I_3 + I_4 .\label{equality2}
    \end{align}
From equalities \eqref{equality1} and \eqref{equality2}, we deduce that 
\begin{equation}\label{SommaInt}
   I_1 \leq |I_2|  + |I_3| + |I_4| +|I_5|.
\end{equation} 
By virtue of assumption \eqref{A1}, we infer
    \begin{align}\label{I1}
        I_1 \geq \delta \nu \int_\Omega \eta^2 |\tau_h Du|^2 \, |\tau_h u|^{\delta-1} dx.
    \end{align}
Now we consider the term $|I_2|$. By using hypothesis \eqref{A2}, Young's  inequality, the properties of $\eta$  we have
\begin{align}\label{I2}
    I_2 &\leq 2L \int_\Omega \eta \, |\tau_h Du|  |D\eta| |\tau_h u|^\delta \notag \\
     & \leq \delta \frac{\nu}{4} \int_\Omega \eta^2 |\tau_h Du|^2 |\tau_h u |^{\delta-1} dx + \frac{c(\nu, L, n, \delta)}{(t-s)^2} \int_{B_t} |\tau_h u |^{\delta+1} dx \notag \\
    & \leq \delta \frac{\nu}{4} \int_\Omega \eta^2 |\tau_h Du|^2 |\tau_h u |^{\delta-1} dx + \frac{c(\nu, L, n, \delta)}{(t-s)^2}|h|^{\delta+1} \int_{B_{t'}} |Du|^{\delta+1} dx,
    \end{align}
where  in the last inequality  we used Lemma \ref{ldiff}. From condition \eqref{A3}, Young's inequality and  the properties of $\eta$, we derive for $ 2^{-k-1}\frac{t' - \sigma}{40} \leq |h|< 2^{-k}\frac{t' - \sigma}{40}$, $k \in \mathbb{N},$
    \begin{align}\label{I3prima}
        I_3 &\leq \delta |h|^\alpha \int_\Omega \eta^2 \,  \left( g_k(x+h)+ g_k(x)\right) |\tau_h u|^{\delta-1}\, |\tau_h Du| \,     (\mu^2+|Du|^2)^{\frac{1}{2}}  \, dx \notag \\
        & \leq \delta \frac{\nu}{4}\int_\Omega \eta^2 \,  |\tau_h u|^{\delta-1}\, |\tau_h Du|^2  \, dx \notag \\
        & \qquad + c(\nu) \delta |h|^{2\alpha} \int_{B_t}    \left( g_k(x+h)+ g_k(x)\right)^2 |\tau_h u|^{\delta-1}\,   (1+|Du|^2)  \, dx \notag \\
        & =: \delta \frac{\nu}{4}\int_\Omega \eta^2 \,  |\tau_h u|^{\delta-1}\, |\tau_h Du|^2  \, dx + c(\nu)\delta |h|^{2 \alpha} I'_3.
    \end{align}
 Using H\"older's Inequality  and Lemma \ref{ldiff}, we get
    \begin{align}\label{I'3}
        I'_3 & \leq \left( \int_{B_t}\left( g_k(x+h)+g_k(x) \right)^{\frac{n}{\alpha}} dx \right)^{\frac{2\alpha}{n}} \, \left( \int_{B_{t}} (1+|Du|^{\frac{n(\delta+1)}{n-2\alpha}} )  \, dx\right)^{\frac{n-2\alpha}{n} \frac{2}{\delta+1}}\, \left( \int_{B_{t}}|\tau_h u|^{\frac{n(\delta+1)}{n-2\alpha}}  dx\right)^{\frac{n-2\alpha}{n} \frac{\delta-1}{\delta+1}} \notag \\
           & \leq c(n, \delta,\alpha)|h|^{\delta-1}\left( \int_{B_t}\left( g_k(x+h)+g_k(x) \right)^{\frac{n}{\alpha}} dx \right)^{\frac{2\alpha}{n}} \, \left( \int_{B_{t}} (1+|Du|^{\frac{n(\delta+1)}{n-2\alpha}})  \, dx\right)^{\frac{n-2\alpha}{n} \frac{2}{\delta+1}}\, \notag  \\
           &\qquad \cdot \left( \int_{B_{t'}}|D u|^{\frac{n(\delta+1)}{n-2\alpha}}  dx\right)^{\frac{n-2\alpha}{n} \frac{\delta-1}{\delta+1}} \notag \\
           &\leq c(n, \delta,\alpha) |h|^{\delta-1}\left( \int_{B_t}\left( g_k(x+h)+g_k(x) \right)^{\frac{n}{\alpha}} dx \right)^{\frac{2\alpha}{n}} \, \left( \int_{B_{t'}}  (1+|Du|^{\frac{n(\delta+1)}{n-2\alpha}})   \, dx\right)^{\frac{n-2\alpha}{n}}.
    \end{align}
Hence, it follows that
 \begin{align}\label{I3}
        I_3 &\leq \delta \frac{\nu}{4}\int_\Omega \eta^2 \,  |\tau_h u|^{\delta-1}\, |\tau_h Du|^2  \, dx \notag \\
        & \qquad +  c(n,\nu, \delta,\alpha) |h|^{\delta-1+2\alpha}\left( \int_{B_t} \left(g_k(x+h)+g_k(x)\right)^{\frac{n}{\alpha}} dx \right)^{\frac{2\alpha}{n}}   \notag \\
        & \qquad \cdot \left( \int_{B_{t'}} (1+|Du|^{\frac{n(\delta+1)}{n-2\alpha}} )   \, dx\right)^{\frac{n-2\alpha}{n}}.
        \end{align}
   From \eqref{A3}, Young's inequality, the properties of $\eta$ and Lemma \ref{ldiff}, we get for $ 2^{-k-1}\frac{t' - \sigma}{40} \leq |h|< 2^{-k}\frac{t' - \sigma}{40}$, $ k \in \mathbb{N}$
   \begin{align}\label{I4prima}
        I_4 &\leq 2|h|^\alpha \int_\Omega \eta \,  \left( g_k(x+h)+ g_k(x)\right) |D \eta| \, |\tau_h u|^{\delta}\,  ( \mu^2+|Du|^2)^\frac{1}{2}  \, dx \notag \\
        & \leq c\int_{B_t} |D\eta|^2 \,  |\tau_h u|^{\delta+1}\, dx \notag \\
        & \qquad + c  \int_{B_t}   |h|^{2\alpha} \left( g_k(x+h)+ g_k(x)\right)^2 |\tau_h u|^{ \delta-1 }\,  (1+ |Du|^2 )   \, dx \notag \\
         & \leq \frac{c}{(t-s)^2}\int_{B_t}  \,  |\tau_h u|^{\delta+1}\, dx  +c|h|^{2\alpha} I'_3 \notag \\
            & \leq \frac{c}{(t-s)^2}  |h|^{\delta+1}  \int_{B_t}  \,   |D u|^{\delta+1}  \, dx  +c|h|^{2\alpha} I'_3,
    \end{align}
    where $I'_3$  has been defined at \eqref{I'3}.\\ 
Inserting \eqref{I'3} in \eqref{I4prima}, we deduce
 \begin{align}\label{I4}
        I_4 & \leq \frac{c}{(t-s)^2} |h|^{\delta+1} \int_{B_t}  \,  |D u|^{\delta+1}\, dx  \notag \\
        & \qquad + c(n, \delta,\alpha) |h|^{\delta-1+2\alpha}\left( \int_{B_t} \left(g_k(x+h)+g_k(x)\right)^{\frac{n}{\alpha}} dx \right)^{\frac{2\alpha}{n}}   \, \left( \int_{B_{t'}}  (1+|Du|^{\frac{n(\delta+1)}{n-2\alpha}})  \, dx\right)^{\frac{n-2\alpha}{n}}.
    \end{align}
   \red Using H\"older's inequality with exponents $ \left( \frac{n}{\alpha}, \frac{n}{n-\alpha}\right) $, the properties of $\eta$ and Lemma \ref{ldiff}, we have
    \begin{align}\label{I5.}
        I_5 &\leq \left( \int_{B_R} |f|^\frac{n}{\alpha}dx \right)^\frac{\alpha}{n} \left( \int_\Omega |\tau_{-h}(\eta^2 |\tau_h u|^{\delta-1} \tau_h u)|^\frac{n}{n-\alpha} dx \right)^\frac{n-\alpha}{n} \notag\\
        &\leq c(n,\alpha) |h|\left( \int_{B_R} |f|^\frac{n}{\alpha}dx \right)^\frac{\alpha}{n} \left( \int_\Omega |D(\eta^2 |\tau_h u|^{\delta-1} \tau_h u)|^\frac{n}{n-\alpha} dx \right)^\frac{n-\alpha}{n} \notag\\
        &\leq c(n,\alpha) |h|\left( \int_{B_R} |f|^\frac{n}{\alpha}dx \right)^\frac{\alpha}{n} \Biggl[\left( \int_\Omega |\eta D \eta|^\frac{n}{n-\alpha} |\tau_h u|^\frac{n\delta}{n-\alpha} dx \right)^\frac{n-\alpha}{n} \notag\\
        & \qquad \qquad \qquad  + \left( \int_\Omega \eta^\frac{2n}{n-\alpha} |\tau_h u|^\frac{n(\delta-1)}{n-\alpha} |\tau_h Du|^\frac{n}{n-\alpha}dx \right)^\frac{n-\alpha}{n} \Biggr]\notag\\
        &\leq c(n,\alpha) |h|\left( \int_{B_R} |f|^\frac{n}{\alpha}dx \right)^\frac{\alpha}{n} \Biggl[\dfrac{1}{t-s}\left( \int_{B_t} |\tau_h u|^\frac{n\delta}{n-\alpha} dx \right)^\frac{n-\alpha}{n} \notag\\
        & \qquad \qquad \qquad  + \left( \int_\Omega \eta^\frac{2n}{n-\alpha} |\tau_h u|^\frac{n(\delta-1)}{2(n-\alpha)}  |\tau_h u|^\frac{n(\delta-1)}{2(n-\alpha)} |\tau_h Du|^\frac{n}{n-\alpha}dx \right)^\frac{n-\alpha}{n} \Biggr],
    \end{align}
    Now, we estimate the second and third integral in the right-hand side of \eqref{I5.}. Since by assumption $Du \in L^\frac{n(\delta+1)}{n-2\alpha}_{loc}(\Omega)$, from H\"older's inequality and Lemma \ref{ldiff} we infer
    \begin{align}\label{I5..}
        \left( \int_{B_t} |\tau_h u|^\frac{n\delta}{n-\alpha} dx \right)^\frac{n-\alpha}{n} & \le c(n,\alpha,\delta) R^\frac{n+\alpha(\delta-1)}{\delta+1} \left( \int_{B_t} |\tau_h u|^\frac{n(\delta+1)}{n-2\alpha} dx \right)^\frac{\delta(n-2\alpha)}{n(\delta+1)} \notag\\
        & \le c(n,\alpha,\delta) R^\frac{n+\alpha(\delta-1)}{\delta+1}  |h|^\delta \left( \int_{B_{t'}} |D u|^\frac{n(\delta+1)}{n-2\alpha} dx \right)^\frac{\delta(n-2\alpha)}{n(\delta+1)}.
    \end{align}
    An application of H\"older's inequality with exponents $ \left( \frac{2(n-\alpha)}{n}, \frac{2(n-\alpha)}{n-2\alpha}\right) $, the properties of $\eta$, $Du \in L^\frac{n(\delta+1)}{n-2\alpha}_{loc}(\Omega)$ and Lemma \ref{ldiff} give that
    \begin{align}\label{I5...}
        &\left( \int_\Omega \eta^\frac{2n}{n-\alpha} |\tau_h u|^\frac{n(\delta-1)}{2(n-\alpha)}  |\tau_h u|^\frac{n(\delta-1)}{2(n-\alpha)} |\tau_h Du|^\frac{n}{n-\alpha}dx \right)^\frac{n-\alpha}{n}  \notag\\
        & \le \left( \int_\Omega \eta^\frac{2n}{n-\alpha}  |\tau_h u|^{ \delta-1} |\tau_h Du|^2dx \right)^\frac{1}{2} \left( \int_\Omega \eta^\frac{2n}{n-\alpha} |\tau_h u|^\frac{n(\delta-1)}{n-2\alpha}  dx \right)^\frac{n-2\alpha}{2n} \notag\\
        & \le c(n,\alpha,\delta) R^\frac{n-2 \alpha}{\delta+1} \left( \int_\Omega \eta^2  |\tau_h u|^{ \delta-1} |\tau_h Du|^2dx \right)^\frac{1}{2} \left( \int_{B_t}  |\tau_h u|^\frac{n(\delta+1)}{n-2\alpha}  dx \right)^\frac{(\delta-1)(n-2\alpha)}{2n(\delta+1)} \notag\\
        & \le c(n,\alpha,\delta) R^\frac{n-2 \alpha}{\delta+1} |h|^\frac{\delta-1}{2} \left( \int_\Omega \eta^2  |\tau_h u|^{ \delta-1} |\tau_h Du|^2dx \right)^\frac{1}{2} \left( \int_{B_{t'}}  |D u|^\frac{n(\delta+1)}{n-2\alpha}  dx \right)^\frac{(\delta-1)(n-2\alpha)}{2n(\delta+1)}.
    \end{align}
Putting \eqref{I5..} and \eqref{I5...} in \eqref{I5.} and then applying Young's inequality, we find that
\begin{align}\label{I5}
    |I_5| & \le  \dfrac{c(n,\alpha,\delta)}{t-s}  R^\frac{n+\alpha(\delta-1)}{\delta+1}|h|^{\delta+1}\left( \int_{B_R} |f|^\frac{n}{\alpha}dx \right)^\frac{\alpha}{n}   \left( \int_{B_{t'}} |D u|^\frac{n(\delta+1)}{n-2\alpha} dx \right)^\frac{\delta(n-2\alpha)}{n(\delta+1)} \notag\\
    & \qquad +  c(n,\alpha,\delta)  R^\frac{n-2 \alpha}{\delta+1} |h|^\frac{\delta+1}{2}\left( \int_{B_R} |f|^\frac{n}{\alpha}dx \right)^\frac{\alpha}{n}\left( \int_\Omega \eta^2  |\tau_h u|^{ \delta-1} |\tau_h Du|^2dx \right)^\frac{1}{2} \left( \int_{B_{t'}}  |D u|^\frac{n(\delta+1)}{n-2\alpha}  dx \right)^\frac{(\delta-1)(n-2\alpha)}{2n(\delta+1)} \notag\\
    & \le  \dfrac{c(n,\alpha,\delta)}{t-s}  R^\frac{n+\alpha(\delta-1)}{\delta+1}|h|^{\delta+1}\left( \int_{B_R} |f|^\frac{n}{\alpha}dx \right)^\frac{\alpha}{n}   \left( \int_{B_{t'}} |D u|^\frac{n(\delta+1)}{n-2\alpha} dx \right)^\frac{\delta(n-2\alpha)}{n(\delta+1)} \notag\\
    & \qquad +  c(n,\alpha,\nu,\delta)  R^\frac{2(n-2 \alpha)}{\delta+1} |h|^{\delta+1}\left( \int_{B_R} |f|^\frac{n}{\alpha}dx \right)^\frac{2\alpha}{n}\left( \int_{B_{t'}}  |D u|^\frac{n(\delta+1)}{n-2\alpha}  dx \right)^\frac{(\delta-1)(n-2\alpha)}{n(\delta+1)} \notag\\
    & \qquad + \delta \frac{\nu}{4} \int_\Omega \eta^2  |\tau_h u|^{ \delta-1} |\tau_h Du|^2dx  .
\end{align}
    \nc
 Inserting \eqref{I1}, \eqref{I2}, \eqref{I3}, \eqref{I4} and \eqref{I5} in \eqref{SommaInt}, we get
    \begin{align}\label{Unione2}
        &\delta \nu \int_\Omega \eta^2 |\tau_h Du|^2 \, |\tau_h u|^{\delta-1} dx \notag\\ & \leq 
         \delta \red \frac{3\nu}{4} \nc \int_\Omega \eta^2 |\tau_h Du|^2 |\tau_h u |^{\delta-1} dx + \frac{c(\nu, L, n, \delta)}{(t-s)^2}|h|^{\delta+1} \int_{B_{t'}} |Du|^{\delta+1} dx \notag \\
        & \qquad +  c(n, \nu,\delta,\alpha) |h|^{\delta-1+2\alpha}\left( \int_{B_t} \left(g_k(x+h)+g_k(x)\right)^{\frac{n}{\alpha}} dx \right)^{\frac{2\alpha}{n}}   \notag \\
        & \qquad \cdot \left( \int_{B_{t'}} (1+|Du|^{\frac{n(\delta+1)}{n-2\alpha}})  \, dx\right)^{\frac{n-2\alpha}{n}} \notag \\
        & \qquad + \red \dfrac{c(n,\alpha,\delta)}{t-s}  R^\frac{n+\alpha(\delta-1)}{\delta+1}|h|^{\delta+1}\left( \int_{B_R} |f|^\frac{n}{\alpha}dx \right)^\frac{\alpha}{n}   \left( \int_{B_{t'}} |D u|^\frac{n(\delta+1)}{n-2\alpha} dx \right)^\frac{\delta(n-2\alpha)}{n(\delta+1)} \nc \notag\\
    & \qquad + \red  c(n,\alpha,\nu,\delta)  R^\frac{2(n-2 \alpha)}{\delta+1} |h|^{\delta+1}\left( \int_{B_R} |f|^\frac{n}{\alpha}dx \right)^\frac{2\alpha}{n}\left( \int_{B_{t'}}  |D u|^\frac{n(\delta+1)}{n-2\alpha}  dx \right)^\frac{(\delta-1)(n-2\alpha)}{n(\delta+1)}.\nc
    \end{align}
 Reabsorbing the first integral on the right-hand side of \eqref{Unione2} by the left-hand side and using that $\eta=1$ on $B_s$, we obtain
   \begin{align}
   &\int_{B_s} |\tau_h Du|^2 \, |\tau_h u|^{\delta-1} dx \notag\\ & \leq 
          \frac{c(\nu, L, n, \delta)}{(t-s)^2}|h|^{\delta+1} \int_{B_{t'}} |Du|^{\delta+1} dx \notag \\
        & \qquad +  c(n, \nu,\delta,\alpha) |h|^{\delta-1+2\alpha}\left( \int_{B_t} \left(g_k(x+h)+g_k(x)\right)^{\frac{n}{\alpha}} dx \right)^{\frac{2\alpha}{n}}  \notag \\
        & \qquad \cdot \left( \int_{B_{t'}}  (1+|Du|^{\frac{n(\delta+1)}{n-2\alpha}} )  \, dx\right)^{\frac{n-2\alpha}{n}} \notag \\
        & \qquad + \red \dfrac{c(n,\alpha,\nu,\delta)}{t-s}  R^\frac{n+\alpha(\delta-1)}{\delta+1}|h|^{\delta+1}\left( \int_{B_R} |f|^\frac{n}{\alpha}dx \right)^\frac{\alpha}{n}   \left( \int_{B_{t'}} |D u|^\frac{n(\delta+1)}{n-2\alpha} dx \right)^\frac{\delta(n-2\alpha)}{n(\delta+1)} \nc \notag\\
    & \qquad + \red c(n,\alpha,\nu,\delta)  R^\frac{2(n-2 \alpha)}{\delta+1} |h|^{\delta+1}\left( \int_{B_R} |f|^\frac{n}{\alpha}dx \right)^\frac{2\alpha}{n}\left( \int_{B_{t'}}  |D u|^\frac{n(\delta+1)}{n-2\alpha}  dx \right)^\frac{(\delta-1)(n-2\alpha)}{n(\delta+1)}. \nc \label{Unione1}
    \end{align}
    Exploiting the identity 
    $$\left| D\left( |\tau_h u|^{\frac{\delta-1}{2}} \, \tau_h u \right) \right|^2 =\left( \frac{\delta + 1}{2} \right)^2 |\tau_h Du|^2 \, |\tau_h u|^{\delta-1}$$
    and Lemma \ref{ldiff}, it follows that
    \begin{align}
        \int_{B_\rho} \left| \tau_\lambda \left( |\tau_h u|^{\frac{\delta-1}{2}} \, \tau_h u \right) \right|^2 dx \leq & \, c \, |\lambda|^2 \int_{B_s} \left| D \left( |\tau_h u|^{\frac{\delta-1}{2}} \, \tau_h u \right) \right|^2 dx \notag\\
        = & \, c \, |\lambda|^2(\delta +1)^2 \int_{B_s} |\tau_h Du|^2 \, |\tau_h u|^{\delta-1} . \label{stima tau}
    \end{align}
Choosing $\lambda= h$,  which is legitimate by our choice of the step size $|h|< \frac{t'- \sigma}{40}$  and putting \eqref{Unione1} in \eqref{stima tau}, we derive 
\begin{align}
    &\int_{B_\rho} \left| \tau_h \left( |\tau_h u|^{\frac{\delta-1}{2}} \, \tau_h u \right) \right|^2 dx \notag\\
    & \leq \dfrac{c(\nu, L, n, \delta)}{(t-s)^2}|h|^{\delta+3}\int_{B_{t'}} |Du|^{\delta+1} dx \notag \\
        & \qquad +  {c(n, \nu,\delta,\alpha)}|h|^{\delta+1+2\alpha} \left( \int_{B_t} \left(g_k(x+h)+g_k(x)\right)^{\frac{n}{\alpha}} dx \right)^{\frac{2\alpha}{n}}  \notag \\
        & \qquad \cdot \left( \int_{B_{t'}} (1+|Du|^{\frac{n(\delta+1)}{n-2\alpha}})  \, dx\right)^{\frac{n-2\alpha}{n}} \notag \\
         & \qquad + \red \dfrac{c(n,\alpha,\nu,\delta)}{t-s}  R^\frac{n+\alpha(\delta-1)}{\delta+1}|h|^{\delta+3}\left( \int_{B_R} |f|^\frac{n}{\alpha}dx \right)^\frac{\alpha}{n}   \left( \int_{B_{t'}} |D u|^\frac{n(\delta+1)}{n-2\alpha} dx \right)^\frac{\delta(n-2\alpha)}{n(\delta+1)} \nc \notag\\
    & \qquad + \red c(n,\alpha,\nu,\delta)  R^\frac{2(n-2 \alpha)}{\delta+1} |h|^{\delta+3}\left( \int_{B_R} |f|^\frac{n}{\alpha}dx \right)^\frac{2\alpha}{n}\left( \int_{B_{t'}}  |D u|^\frac{n(\delta+1)}{n-2\alpha}  dx \right)^\frac{(\delta-1)(n-2\alpha)}{n(\delta+1)}. \nc \label{tau tau}
\end{align}
 An application of Lemma \ref{D1} with $p=\delta+1$  and the fact that $\delta+1 \geq 2$ give  that  
\begin{equation}
    |\tau_h (\tau_h u)|^{\delta+1} \leq  \left| \tau_h \left( |\tau_h u|^{\frac{\delta-1}{2}}  \, \tau_h u \right) \right|^2 . \label{tau est}
\end{equation}
 Combining \eqref{tau tau} and \eqref{tau est} and then dividing both sides by $|h|^{\delta+1+2 \alpha}$, we get
\begin{align}
    &\int_{B_\rho} \frac{|\tau_h (\tau_h u)|^{\delta+1}}{|h|^{\delta+1+2 \alpha}}  dx \notag\\
    & \leq \dfrac{c(\nu, L, n, \delta)}{(t-s)^2}|h|^{2(1-\alpha)}\int_{B_{t'}} |Du|^{\delta+1} dx \notag \\
        & \qquad +  {c(n, \nu,\delta,\alpha)} \left( \int_{B_t} \left(g_k(x+h)+g_k(x)\right)^{\frac{n}{\alpha}} dx \right)^{\frac{2\alpha}{n}}  \notag \\
        & \qquad \cdot \left( \int_{B_{t'}}(1+|Du|^{\frac{n(\delta+1)}{n-2\alpha}}) \, dx\right)^{\frac{n-2\alpha}{n}} \notag \\
         & \qquad + \red \dfrac{c(n,\alpha,\nu,\delta)}{t-s}  R^\frac{n+\alpha(\delta-1)}{\delta+1}|h|^{2(1-\alpha)}\left( \int_{B_R} |f|^\frac{n}{\alpha}dx \right)^\frac{\alpha}{n}   \left( \int_{B_{t'}} |D u|^\frac{n(\delta+1)}{n-2\alpha} dx \right)^\frac{\delta(n-2\alpha)}{n(\delta+1)} \nc \notag\\
    & \qquad + \red c(n,\alpha,\nu,\delta)  R^\frac{2(n-2 \alpha)}{\delta+1} |h|^{2(1-\alpha)}\left( \int_{B_R} |f|^\frac{n}{\alpha}dx \right)^\frac{2\alpha}{n}\left( \int_{B_{t'}}  |D u|^\frac{n(\delta+1)}{n-2\alpha}  dx \right)^\frac{(\delta-1)(n-2\alpha)}{n(\delta+1)}. \nc \label{tau tau 1}
\end{align}
Now, raise to the power $\frac{p}{2}$,  with $p$ appearing in \eqref{gk}, integrate over the ball $B(0,(t' - \sigma)/40 )$ w.r.t.\ the measure $\frac{d h}{|h|^n}.$
Since the functions $g_k$ are defined for $2^{-k-1}\frac{t' - \sigma}{40} \leq |h|< 2^{-k}\frac{t' - \sigma}{40}$, we interpret the ball $B(0,(t' - \sigma)/40 )$ as
$$ B(0,(t' - \sigma)/40 )= \bigcup_{k=0}^{\infty} B(0,2^{-k}(t' - \sigma)/40) \setminus B(0,2^{-k-1}(t' - \sigma)/40)=: \bigcup_{k=0}^{\infty} E_k.$$
 Then, by \eqref{tau tau 1}, we infer the following estimate 
\begin{align}
 &\int_{ B_\frac{\rho-\rho'}{8}(0)} \left(  \int_{B_\rho} \frac{|\tau_h (\tau_h u)|^{\delta+1} }{|h|^{\delta +2\alpha+1}} dx \right)^{\frac{p}{2}} \frac{dh}{|h|^n} \notag\\
 & \leq  \frac{c}{(t-s)^{p}} \int_{B_\frac{t'-\sigma}{40}(0)} \left( \int_{B_{t'}} |Du|^{\delta+1} dx \right)^{{\frac{p}{2}}} |h|^{{(1-\alpha)p}}  \frac{dh}{|h|^n}\notag \\
        & \qquad +  c \sum_{k=0}^\infty \int_{E_k}\left( \int_{B_t} \left(g_k(x+h)+g_k(x)\right)^{\frac{n}{\alpha}} dx  \right)^{\frac{p\alpha}{n}}  \frac{dh}{|h|^n} \notag \\
        & \qquad \cdot \left( \int_{B_{t'}} (1+|Du|^{\frac{n(\delta+1)}{n-2\alpha}})  \, dx\right)^{\frac{(n-2\alpha)p}{2n}} \notag \\
        & \qquad + \red \dfrac{c}{(t-s)^{p/2}}  R^\frac{p(n+\alpha(\delta-1))}{2(\delta+1)}\left( \int_{B_R} |f|^\frac{n}{\alpha}dx \right)^\frac{\alpha p}{2n}   \left( \int_{B_{t'}} |D u|^\frac{n(\delta+1)}{n-2\alpha} dx \right)^\frac{\delta(n-2\alpha)p}{2n(\delta+1)} \int_{B_\frac{t'-\sigma}{40}(0)}  |h|^{p(1-\alpha)} \dfrac{d h}{|h|^n}\nc \notag\\
    & \qquad + \red c  R^\frac{p(n-2 \alpha)}{\delta+1} \left( \int_{B_R} |f|^\frac{n}{\alpha}dx \right)^\frac{\alpha p}{n}\left( \int_{B_{t'}}  |D u|^\frac{n(\delta+1)}{n-2\alpha}  dx \right)^\frac{(\delta-1)(n-2\alpha)p}{2n(\delta+1)} \int_{B_\frac{t'-\sigma}{40}(0)}  |h|^{p(1-\alpha)} \dfrac{d h}{|h|^n} \nc, \label{norm!}
\end{align}
with $c=c(n, \delta,\alpha, \nu, L,p)$ positive constant.  Here, in the integral on the left-hand side, we used the fact that $B(0,(t' - \sigma)/40 )= B(0,(\rho-\rho')/8 )$, which follows from the definition of radii given in \eqref{radii}.

Now, we take care of the term
$$J=\sum_{k=0}^\infty \int_{E_k}\left( \int_{B_t} \left(g_k(x+h)+g_k(x)\right)^{\frac{n}{\alpha}} dx  \right)^{\frac{p\alpha}{n}}  \frac{dh}{|h|^n}. $$
Following the arguments in \cite{Baison},
we write the right-hand side of the previous estimate in polar coordinates, so $h \in E_k$ if, and only if, $h= m \xi$ for some $ 2^{-k-1}\frac{t' - \sigma}{40} \leq m< 2^{-k}\frac{t' - \sigma}{40}$ and some $\xi$ in the unit sphere $\mathbb{S}^{n-1}$ on $\mathbb{R}^n$. We denote by $d S(\xi)$ the surface measure on $\mathbb{S}^{n-1}$.  Setting $\Delta_{m \xi} g_k(x)=g_k(x+m \xi)$ and $ m_k=2^{-k}(t' - \sigma)/40 $,  we infer
\begin{align*}
J \leq &  \displaystyle\sum_{k=0}^{\infty} \displaystyle\int_{m_{k+1}}^{m_k} \displaystyle\int_{\mathbb{S}^{n-1}}  \biggl( \displaystyle\int_{B_{t}}(g_k(x+m \xi)+g_k(x))^\frac{n}{\alpha} dx \biggr)^{\frac{p \alpha}{n}} dS(\xi)\dfrac{dm}{m} \notag\\
= & \displaystyle\sum_{k=0}^{\infty} \displaystyle\int_{m_{k+1}}^{m_k} \displaystyle\int_{\mathbb{S}^{n-1}}  \Vert \Delta_{m \xi}g_k+g_k \Vert_{L^{\frac{n}{\alpha}}(B_{t})}^{p} dS(\xi)\dfrac{dm}{m}.
\end{align*}
  We note that for each $\xi \in \mathbb{S}^{n-1}$ and $m_{k-1}\leq m \leq m_k$
\begin{align*}
\Vert \Delta_{m \xi}g_k+g_k \Vert_{L^{\frac{n}{\alpha}}(B_{t})}  \leq \,& \Vert g_k \Vert_{L^{\frac{n}{\alpha}}(B_{t-m_k \xi})} + \Vert g_k \Vert_{L^{\frac{n}{\alpha}}(B_{t})} \notag\\
\leq &\, 2  \Vert g_k \Vert_{L^{\frac{n}{\alpha}}(B_{R})},
\end{align*}
hence
\begin{equation}
    J \leq  c(n)\,2^p \,\text{log }2\, \sum_{k=0}^\infty \Vert g_k \Vert_{L^\frac{n}{\alpha}(B_R)}^p , \label{J}
\end{equation}
which is finite by assumption \eqref{gk}.

Inserting \eqref{J} in \eqref{norm!}, we find that
\begin{align}
& \int_{  B_\frac{\rho-\rho'}{8}(0)} \left(  \int_{B_\rho} \frac{|\tau_h (\tau_h u)|^{\delta+1} }{|h|^{\delta +2\alpha+1}} dx \right)^{\frac{p}{2}} \frac{dh}{|h|^n} \notag\\
& \leq  \frac{c}{(t-s)^{p}} \left( \int_{B_{t'}} |Du|^{\delta+1} dx \right)^{{\frac{p}{2}}} \int_{B_R(0)}  |h|^{{(1-\alpha)p}}  \frac{dh}{|h|^n}\notag \\
        & \qquad +  c \sum_{k=0}^\infty \Vert g_k \Vert_{L^\frac{n}{\alpha}(B_R)}^p\left( \int_{B_{t'}} (1+|Du|^{\frac{n(\delta+1)}{n-2\alpha}})  \, dx\right)^{\frac{(n-2\alpha)p}{2n}} \notag \\
         & \qquad + \red \dfrac{c}{(t-s)^{p/2}}  R^\frac{p(n+\alpha(\delta-1))}{2(\delta+1)}\left( \int_{B_R} |f|^\frac{n}{\alpha}dx \right)^\frac{\alpha p}{2n}   \left( \int_{B_{t'}} |D u|^\frac{n(\delta+1)}{n-2\alpha} dx \right)^\frac{\delta(n-2\alpha)p}{2n(\delta+1)} \int_{B_R(0)}  |h|^{p(1-\alpha)} \dfrac{d h}{|h|^n}\nc \notag\\
    & \qquad + \red c  R^\frac{p(n-2 \alpha)}{\delta+1} \left( \int_{B_R} |f|^\frac{n}{\alpha}dx \right)^\frac{\alpha p}{n}\left( \int_{B_{t'}}  |D u|^\frac{n(\delta+1)}{n-2\alpha}  dx \right)^\frac{(\delta-1)(n-2\alpha)p}{2n(\delta+1)} \int_{B_R(0)}  |h|^{p(1-\alpha)} \dfrac{d h}{|h|^n} \nc. \label{norm!!}
\end{align}
We observe that, since $\alpha \in (0,1)$, the integral $$\int_{B_R(0)}|h|^{{p(1-\alpha)}} \frac{dh}{|h|^n} = c(n,\alpha,p, R)$$
is finite. 

 Now, since $p \in [1, \frac{2n}{n-2 \alpha}]$, we can use Corollary \ref{corollary} with $\delta+1$, $\frac{p(\delta+1)}{2}$ and $\frac{2 \alpha}{\delta+1}$ in place of $p$, $q$ and $\alpha$ respectively. This, together with \eqref{norm!!} and \eqref{radii}, gives for every $r<\sigma<t'<R$ the following estimate 
\begin{align}
& \left(  \int_{ B_{\sigma}} |Du|^{\frac{n(\delta+1)}{n-2\alpha}} dx \right)^{\frac{p(n-2\alpha)}{2n}} \notag\\
& \leq   c \left( \int_{B_{R}}  |u|^{\delta+1} dx \right)^{{\frac{p}{2}}}
    +\frac{c}{(t'-\sigma)^{p(1+\alpha)}}  \left( \int_{B_{R}}  |u|^{\delta+1} dx \right)^{{\frac{p}{2}}}\notag\\
 & \qquad +c\left( \int_{B_{R}}  |Du|^{\delta+1} dx \right)^{{\frac{p}{2}}} \notag \\
        & \qquad +  \frac{c
    }{(t'-\sigma)^{{\alpha p}}}  \left( \int_{ B_{R}} |Du|^{\delta+1} dx \right)^{{\frac{p}{2}}} \notag \\
        & \qquad + \frac{c}{(t'-\sigma)^p} \left( \int_{B_{R}} |Du|^{\delta+1} dx \right)^{{\frac{p}{2}}} \notag \\
        & \qquad +  \gamma   \sum_{k=0}^\infty \Vert g_k \Vert_{L^\frac{n}{\alpha}(B_R)}^p \left( \int_{B_{t'}}  (1+|Du|^{\frac{n(\delta+1)}{n-2 \alpha}} )  \, dx\right)^{\frac{(n-2\alpha)p}{2n}} \notag \\
       & \qquad + \red \dfrac{c}{(t'-\sigma)^{p/2}}  R^\frac{p(n+\alpha(\delta-1))}{2(\delta+1)}\left( \int_{B_R} |f|^\frac{n}{\alpha}dx \right)^\frac{\alpha p}{2n}   \left( \int_{B_{t'}} |D u|^\frac{n(\delta+1)}{n-2\alpha} dx \right)^\frac{\delta(n-2\alpha)p}{2n(\delta+1)} \nc \notag\\
    & \qquad + \red c  R^\frac{p(n-2 \alpha)}{\delta+1} \left( \int_{B_R} |f|^\frac{n}{\alpha}dx \right)^\frac{\alpha p}{n}\left( \int_{B_{t'}}  |D u|^\frac{n(\delta+1)}{n-2\alpha}  dx \right)^\frac{(\delta-1)(n-2\alpha)p}{2n(\delta+1)}, \nc\label{est1!}
\end{align}
for a positive constants $c=c(n, \delta,\alpha, \nu, L,p)$ and  $\gamma=\gamma(n, \delta,\alpha, \nu, L,p)$.

Then, applying Young's inequality in the  \red last two terms \nc on the right-hand side of \eqref{est1!}, we derive
\begin{align}
& \left(  \int_{ B_{\sigma}} |Du|^{\frac{n(\delta+1)}{n-2\alpha}} dx \right)^{\frac{p(n-2\alpha)}{2n}} \notag\\
& \leq   c \left( \int_{ B_{R}}  |u|^{\delta+1} dx \right)^{{\frac{p}{2}}}
    +\frac{c}{(t'-\sigma)^{p(1+\alpha)}}  \left( \int_{ B_{R}}  |u|^{\delta+1} dx \right)^{{\frac{p}{2}}}\notag\\
 & \qquad +c\left(\int_{B_{R}} |Du|^{\delta+1} dx \right)^{{\frac{p}{2}}} \notag \\
        & \qquad +  \frac{{c}}{(t'-\sigma)^{\alpha p}}   \left(\int_{ B_{R}} |Du|^{\delta+1} dx \right)^{{\frac{p}{2}}}\notag \\
        & \qquad + \frac{{c}}{(t'-\sigma)^{p}}  \left(\int_{B_R} |Du|^{\delta+1} dx \right)^{{\frac{p}{2}}} \notag \\
        & \qquad +  \gamma   \sum_{k=0}^\infty \Vert g_k \Vert_{L^\frac{n}{\alpha}(B_R)}^p \left( \int_{B_{t'}}  |Du|^{\frac{n(\delta+1)}{n-2 \alpha}}  \, dx\right)^{\frac{(n-2\alpha)p}{2n}} \notag \\
        & \qquad +  c \gamma R^\frac{p(n-2 \alpha)}{2}    \sum_{k=0}^\infty \Vert g_k \Vert_{L^\frac{n}{\alpha}(B_R)}^p \notag \\
        & \qquad + \red \frac{1}{8}\left( \int_{B_{t'}} |Du|^{\frac{n(\delta+1)}{n-2 \alpha}} \, dx\right)^{\frac{(n-2\alpha) p}{2n}}+  \dfrac{c}{(t'-\sigma)^{p(\delta+1)/2}}  R^\frac{p(n+\alpha(\delta-1))}{2}\left( \int_{B_R} |f|^\frac{n}{\alpha}dx \right)^\frac{\alpha p(\delta+1)}{2n}  \nc \notag\\
    & \qquad +  \red \frac{1}{8}\left( \int_{B_{t'}} |Du|^{\frac{n(\delta+1)}{n-2 \alpha}} \, dx\right)^{\frac{(n-2\alpha) p}{2n}}+c  R^\frac{p(n-2 \alpha)}{2} \left( \int_{B_R} |f|^\frac{n}{\alpha}dx \right)^\frac{\alpha p(\delta+1)}{2n},  \nc \notag
\end{align}
where $c$ and $ \gamma $  are positive constants depending at most on $(n,\delta, \alpha,\nu, L, p).$

We choose a radius $0<R_\delta\le 1$, depending on $(x_0,n,\delta, \alpha,\nu, L, p),$ such that for $0<R \le R_\delta$ it holds
\begin{equation} \label{smallness}
     \gamma  \sum_{k=0}^\infty \Vert g_k \Vert^p_{L^\frac{n}{\alpha}(B_R)}< \dfrac{1}{4} .
\end{equation}
\noindent We show that such a radius $R_\delta$ does exist. Indeed, by assumption \eqref{gk}, we have that there exists $m \in \mathbb{N}$ such that
$$\sum_{k=m+1}^\infty \Vert g_k \Vert^p_{L^\frac{n}{\alpha}(B_R)} < \dfrac{1}{8  \gamma  }.$$
Moreover, by the absolute continuity of the integral, for every $k \in \{0,\dots,m\}$ there exists $0 < R_{\delta,k} \le 1$ such that if $R \le R_{\delta,k}$, then
$$ \Vert g_k \Vert^p_{L^\frac{n}{\alpha}(B_R)} < \dfrac{1}{8(m+1)  \gamma }.$$
Hence, \eqref{smallness} follows by choosing $R_\delta= \min \{ R_{\delta_k}, \, k=0,\dots,m \}$.

Therefore, if $0< R \leq R_\delta$, we get for every $r<\sigma <t'<R$
\begin{align}
 &\left(  \int_{ B_{\sigma}} |Du|^{\frac{n(\delta+1)}{n-2\alpha}} dx \right)^{\frac{p(n-2\alpha)}{2n}} \notag\\
 & \leq   c \left( \int_{B_{R}}  |u|^{\delta+1} dx \right)^{{\frac{p}{2}}}
    +\frac{c}{(t'-\sigma)^{p(1+\alpha)}}  \left( \int_{B_{R}}  |u|^{\delta+1} dx \right)^{{\frac{p}{2}}}\notag\\
 & \qquad +c\left(\int_{B_{R}} |Du|^{\delta+1} dx \right)^{{\frac{p}{2}}} \notag \\
        & \qquad +  \frac{{c}}{(t'-\sigma)^{\alpha p}}   \left(\int_{ B_{R}} |Du|^{\delta+1} dx \right)^{{\frac{p}{2}}}\notag \\
        & \qquad + \frac{{c}}{(t'-\sigma)^{p}}  \left(\int_{B_R} |Du|^{\delta+1} dx \right)^{{\frac{p}{2}}} \notag \\
        & \qquad +\frac{1}{4} \left( \int_{B_{t'}}   (1+|Du|^{\frac{n(\delta+1)}{n-2 \alpha}})  \, dx\right)^{\frac{(n-2\alpha)p}{2n}} \notag \\
        & \qquad +  c  R^\frac{p(n-2 \alpha)}{2}   \notag \\
        & \qquad + \red \frac{1}{4}\left( \int_{B_{t'}} |Du|^{\frac{n(\delta+1)}{n-2 \alpha}} \, dx\right)^{\frac{(n-2\alpha) p}{2n}}+  \dfrac{c}{(t'-\sigma)^{p(\delta+1)/2}}  R^\frac{p(n+\alpha(\delta-1))}{2}\left( \int_{B_R} |f|^\frac{n}{\alpha}dx \right)^\frac{\alpha p(\delta+1)}{2n}  \nc \notag\\
    & \qquad + \red c  R^\frac{p(n-2 \alpha)}{2} \left( \int_{B_R} |f|^\frac{n}{\alpha}dx \right)^\frac{\alpha p(\delta+1)}{2n} \nc.
\end{align}
% \begin{rmk}
%    $R_\delta$ depends on $\delta$, but this is not an issue since the iteration concludes after a finite number of steps.
%\end{rmk} 
Thus applying the iteration Lemma \ref{lm2}, we infer
\begin{align}
 \left(  \int_{ B_{r} } |Du|^{\frac{n(\delta+1)}{n-2\alpha}} dx \right)^{\frac{p(n-2\alpha)}{2n}}  & \leq   \tilde{c} \left( \int_{ B_{R}}  |u|^{\delta+1} dx \right)^{{\frac{p}{2}}}
    +\frac{\tilde{c}}{(R-r)^{p(1+\alpha)}}  \left( \int_{B_{R}}  |u|^{\delta+1} dx \right)^{{\frac{p}{2}}}\notag\\
 & \qquad +\tilde{c} \left(\int_{ B_{R}} |Du|^{\delta+1} dx \right)^{{\frac{p}{2}}} \notag \\
        & \qquad +  \frac{\Tilde{c}}{(R-r)^{\alpha p}}   \left(\int_{ B_{R}} |Du|^{\delta+1} dx \right)^{{\frac{p}{2}}}\notag \\
        & \qquad + \frac{\Tilde{c}}{(R- r)^{p}}  \left(\int_{B_R} |Du|^{\delta+1} dx \right)^{{\frac{p}{2}}} \notag \\
       & \qquad +  \Tilde{c}  R^\frac{p(n-2 \alpha)}{2}   + \red \dfrac{\Tilde{c}}{(R-r)^{p(\delta+1)/2}}  R^\frac{p(n+\alpha(\delta-1))}{2}\left( \int_{B_R} |f|^\frac{n}{\alpha}dx \right)^\frac{\alpha p(\delta+1)}{2n}  \nc \notag\\
    & \qquad + \red \Tilde{c}  R^\frac{p(n-2 \alpha)}{2} \left( \int_{B_R} |f|^\frac{n}{\alpha}dx \right)^\frac{\alpha p(\delta+1)}{2n} \nc, \notag
\end{align}
for every concentric balls $ B_{r} \subset B_R \subseteq B_{R_\delta} \Subset \Omega $, where $\tilde{c}=\tilde{c}(n, \delta,\alpha, \nu, L,p)$ is a positive constant. Eventually, raising to the power $2/p$, we get the desired estimate.
 \end{proof}

\noindent Theorem \ref{apriori thm} allows us to implement a Moser-type iteration scheme that improves the regularity of a local weak solution to \eqref{equation} from $W^{1,2}_{\textrm{loc}}(\Omega)$ to $W^{1,q}_{\textrm{loc}}(\Omega)$, for every $q \in [2, +\infty)$.

\begin{thm}
Let $0< \alpha <1$ and $1 \le p \le \frac{2n}{n-2\alpha} $. Let $u \in W^{1,2}_{\textrm{loc}}(\Omega)$ be a local weak solution of \eqref{equation} such that $ Du \in L^\frac{2n}{n-2\alpha}_{\textrm{loc}}(\Omega)$. Then, $u \in W^{1,q}_{\textrm{loc}}(\Omega)$ for every $q \in [2,+\infty)$ and for every $x_0 \in \Omega$ there exists a radius $\tilde{R}=\tilde{R}(x_0,n,q,\alpha,\nu,L,p) \le 1$ and a ball $B_{\tilde{R}}(x_0)  \Subset \Omega$ such that we have the following estimate 
\begin{align}
  \left(  \int_{ B_{R/2}(x_0) } |Du|^{q} dx \right)^{\frac{1}{q}}  & \leq  \tilde{c}\Biggl\{ \Vert u\Vert_{L^\infty(B_R(x_0))}+\red \Vert f \Vert_{L^{\frac{n}{\alpha}}( B_{R}(x_0) )} \nc +  \left( \int_{B_{R}(x_0)} |Du|^{2} dx\right)^{\frac{1}{2}}  + 1   \Biggr\}. \label{stima ite}
\end{align}
for every radius $0<R\le \tilde{R} $ and for a positive constant  $\tilde{c}=\tilde{c}(n, q,\alpha, \nu, L,p,R)$.
\end{thm}

\begin{proof}
    We define a sequence of exponents $(\delta_i)_{i \in \mathbb{N}_0}$ by
    \begin{equation}
        \begin{cases}
            \delta_0=1, \\
    \delta_{i+1}=\dfrac{n(\delta_i+1)}{n-2\alpha}-1=2 \left(  \dfrac{n}{n-2\alpha}\right)^{i+1}-1, \qquad i \ge 0.
        \end{cases} \notag
    \end{equation}
    Clearly $\delta_i \to + \infty$ as $i \to +\infty$. Hence, for all $q \in [2,+\infty)$ there exists $i_* \in \mathbb{N}$ such that $\delta_{i_*-1}+1< q \le\delta_{i_*}+1$.

    For any $i \in \{ 0,1,\dots,i_* \}$, let $0 < R_{\delta_i} \le 1$ be the radius  such that \eqref{smallness} holds. Now, fix a radius $R$ such that $R < \min \{ R_{\delta_i}: i =0,1, \dots, i_* \}$ and define a sequence of radii $(\rho_i)_{i \in \mathbb{N}_0}$ by
    \begin{equation}
        \rho_i = \dfrac{R}{2} \left(   1+\dfrac{1}{2^i}\right). \notag
    \end{equation}
By  the a priori assumption $Du \in L^\frac{2n}{n-2 \alpha}_{\mathrm{loc}}(\Omega)$, we are legitimate to use estimate \eqref{stima a priori} with $\delta=\delta_{i-1}$, $r=\rho_i$ and $ R=\rho_{i-1}$, thus getting
   \begin{align}
  \left(  \int_{ B_{\rho_i}(x_0) } |Du|^{\delta_i+1} dx \right)^{\frac{\delta_{i-1}+1}{\delta_i +1}}  & \leq \frac{\tilde{c_i}}{(\rho_{i-1}- \rho_i)^{\red \beta_i \nc}} \Biggl[ \int_{B_{\rho_{i-1}}(x_0)}  |u|^{\delta_{i-1}+1} dx  \notag\\
 & \qquad +\int_{B_{\rho_{i-1}}(x_0)} |Du|^{\delta_{i-1}+1} dx 
      + \red \Vert f \Vert^{\delta_{i-1} + 1}_{L^{\frac{n}{\alpha}}( B_{R}(x_0) )}  \nc + 1  \Biggr]. 
\end{align}
Since $\rho_{i-1}-\rho_{i}= \frac{R}{2^{i+1}}$, we get
  \begin{align}
  \left(  \int_{ B_{\rho_i}(x_0) } |Du|^{\delta_i+1} dx \right)^{\frac{\delta_{i-1}+1}{\delta_i +1}}  & \leq \frac{\tilde{c_i} 2^{(i+1)(1+\alpha)}}{R^{\red \beta_i \nc}} \Biggl[ \int_{B_{\rho_{i-1}}(x_0)}  |u|^{\delta_{i-1}+1} dx  \notag\\
 & \qquad +\int_{B_{\rho_{i-1}}(x_0)} |Du|^{\delta_{i-1}+1} dx 
      +\red \Vert f \Vert^{\delta_{i-1} + 1}_{L^{\frac{n}{\alpha}}( B_{R}(x_0) )}  \nc  + 1  \Biggr]. 
\end{align}
%Without loss of generality, we can assume that $R\leq 1$. 
Define $$K_i=\max \{1,  \tilde{c_i}2^{2(i+1)(1+\alpha)}  \}.$$
Then, we have
 \begin{align}
  \left(  \int_{ B_{\rho_i}(x_0) } |Du|^{\delta_i+1} dx \right)^{\frac{\delta_{i-1}+1}{\delta_i +1}}  & \leq    \frac{K_i}{R^{ \red \beta_i \nc}}\Biggl\{\int_{B_{\rho_{i-1}}(x_0)}  |u|^{\delta_{i-1}+1} dx
    \notag\\
 & \qquad +\int_{B_{\rho_{i-1}}(x_0)} |Du|^{\delta_{i-1}+1} dx   + \,\red \Vert f \Vert^{\delta_{i-1} + 1}_{L^{\frac{n}{\alpha}}( B_{R}(x_0) )}  \nc  + 1  \Biggr\}. 
\end{align}
Raising to the power $\frac{1}{\delta_{i-1}+1}$, we find that
 \begin{align}
  \left(  \int_{ B_{\rho_i}(x_0) } |Du|^{\delta_i+1} dx \right)^{\frac{1}{\delta_i +1}}  & \leq    \frac{K_i^{\frac{1}{\delta_{i-1}+1}} 2^{\frac{2}{\delta_{i-1}+1}}}{R^\frac{\red \beta_i \nc}{\delta_{i-1}+1}}\Biggl\{\left(\int_{B_{\rho_{i-1}}(x_0)}  |u|^{\delta_{i-1}+1} dx \right)^\frac{1}{\delta_{i-1}+1}  
    \notag\\
 & \qquad + \left( \int_{B_{\rho_{i-1}}(x_0)} |Du|^{\delta_{i-1}+1} dx\right)^{\frac{1}{\delta_{i-1}+1}}   + \,\red \Vert f \Vert_{L^{\frac{n}{\alpha}}( B_{R}(x_0) )}  \nc   + 1 \Biggr\}. 
\end{align}
Using the boundedness of $u$ and the fact $\rho_{i-1} \leq R \le 1,$ we infer
 \begin{align}
  \left(  \int_{ B_{\rho_i}(x_0) } |Du|^{\delta_i+1} dx \right)^{\frac{1}{\delta_i +1}}  & \leq    \frac{K_i^{\frac{1}{\delta_{i-1}+1}} 2^{\frac{2}{\delta_{i-1}+1}}}{R^\frac{\red \beta_i \nc}{\delta_{i-1}+1}}\Biggl\{c(n)^\frac{1}{\delta_{i-1}+1} \Vert u\Vert_{L^\infty(B_R(x_0))}
    \notag\\
 & \qquad + \left( \int_{B_{\rho_{i-1}}(x_0)} |Du|^{\delta_{i-1}+1} dx\right)^{\frac{1}{\delta_{i-1}+1}}   + \,\red \Vert f \Vert_{L^{\frac{n}{\alpha}}( B_{R}(x_0) )} \nc + 1  \Biggr\}. 
\end{align}
Setting $$\overline{K_i}=K_i^{\frac{1}{\delta_{i-1}+1}} 2^{\frac{2}{\delta_{i-1}+1}} c(n)^{\frac{2}{\delta_{i-1}+1}}, \qquad \Gamma= \Vert u\Vert_{L^\infty(B_R(x_0))}+\red \Vert f \Vert_{L^{\frac{n}{\alpha}}( B_{R}(x_0) )} \nc + 1 , $$

\noindent the previous estimate reads as
 \begin{align}
  \left(  \int_{ B_{\rho_i}(x_0) } |Du|^{\delta_i+1} dx \right)^{\frac{1}{\delta_i +1}}  & \leq    \frac{\overline{K}_i}{R^\frac{\red \beta_i \nc}{\delta_{i-1}+1}}\Biggl\{ \Gamma +  \left( \int_{B_{\rho_{i-1}}(x_0)} |Du|^{\delta_{i-1}+1} dx\right)^{\frac{1}{\delta_{i-1}+1}}   \Biggr\}. 
\end{align}
Iterating this estimate for $i \in \{ 1, \dots, i_*\}$,  where $i_*$ has been chosen at the beginning of the proof,  we get
\begin{align}
  \left(  \int_{ B_{\rho_{i_*}}(x_0) } |Du|^{\delta_{i_*}+1} dx \right)^{\frac{1}{\delta_{i_*} +1}}  & \leq    \Tilde{K}_R\Biggl\{ \Gamma+  \left( \int_{B_{R}(x_0)} |Du|^{2} dx\right)^{\frac{1}{2}}   \Biggr\},
\end{align}
where $\tilde{K_R}= i_* \prod_{i=1}^{i_*}\frac{\overline{K_i}}{R^{\frac{\red \beta_i \nc}{\delta_{i-1}+1}}}$.
Thanks to H\"older's inequality, we have
\begin{align}
  \left(  \int_{ B_{R/2}(x_0) } |Du|^{q} dx \right)^{\frac{1}{q}}  & \leq  c(n,q,\delta_{i_*},R)  \Tilde{K}_R\Biggl\{ \Gamma+  \left( \int_{B_{R}(x_0)} |Du|^{2} dx\right)^{\frac{1}{2}}   \Biggr\},
\end{align}
that is the conclusion.
\end{proof}

\section{Proof of Theorem \ref{main thm}}\label{secapp} 
\noindent In this section, we give the proof of Theorem \ref{main thm}.  First, we define a sequence of regular problems with a unique smooth solution $u_\varepsilon$, satisfying the a priori estimates. Then, we show that the solutions $u_\varepsilon$ converge to the original solution of our problem, which subsequently inherits the same regularity properties. 

\proof[Proof of Theorem \ref{main thm}]
Let us fix a non-negative smooth kernel $\phi \in \mathcal{C}^\infty_0(B_1(0))$ such that $\int_{B_1(0)} \phi =1$ and consider the corresponding family of mollifiers $(\phi_\varepsilon)_{\varepsilon >0}$. For a given ball $B_{2R} \Subset \Omega$ and every $0<\varepsilon < \text{dist}(B_R,\partial \Omega)$, we set
$$A_\varepsilon(x,\xi)= \int_{B_\varepsilon(0)} \phi_{\varepsilon}(y) A(x-y,\xi) dy, \quad f_\varepsilon(x)= f * \phi_\varepsilon(x),$$
for a.e.\ $x \in B_R$ and every $\xi \in \mathbb{R}^n$.
One can easily check that $A_\varepsilon(x,\xi)$  is such that $A_\varepsilon(\cdot,0) \in L^2(B_R)$ and  satisfies assumptions \eqref{A1} and \eqref{A2}. Moreover, denoting by $g^\varepsilon_k= g_k * \phi_\varepsilon$, from \eqref{gk} and \eqref{A3} it follows that
\begin{equation}
    \displaystyle\sum_{k=0}^{\infty} \Vert g^\varepsilon_k \Vert^{p}_{L^\frac{n}{\alpha}(B_R)} < \infty \label{gk1}
\end{equation}
and
\begin{equation}\tag{A4}
    |A_\varepsilon(x,\xi)-A_\varepsilon(y, \xi)| \leq |x-y|^{\alpha} (g^\varepsilon_k(x)+g^\varepsilon_k(y))   (\mu^2+|\xi|^2)^\frac{1}{2}, \label{A4}
\end{equation}

\noindent for a.e.\ $x,y \in B_R$ such that $2^{-k} R \leq |x-y| < 2^{-k+1}R$ and for every $\xi \in \mathbb{R}^{ n}$.

Let $u \in W^{1,2}_{\textrm{loc}}(\Omega)$ be a local weak solution to \eqref{equation} and let  $ u_\varepsilon \in u+ W^{1,2}_0(B_R)$ the solution to the Dirichlet problem
\begin{equation}
    \begin{cases}
      -\text{ div}A_\varepsilon(x,Du_\varepsilon)=f_\varepsilon(x) & \qquad \text{in } B_R,  \\
      u_\varepsilon=u & \qquad \text{on } \partial B_R.
    \end{cases} \label{DP}
\end{equation}
By
standard elliptic regularity (see for example \cite[Chapter 8]{giusti}), we know that $u_\varepsilon \in \mathcal{C}^{1,\beta}_{\textrm{loc}}(B_R)$, for some $\beta \in (0,1)$.

Since $u_\varepsilon \in u+ W_0^{1,2}(B_R)$ is a solution to \eqref{DP} and $u \in W^{1,2}(\Omega)$ is a solution to \eqref{equation}, choosing $\varphi=u_\varepsilon-u$ as test function for both equations, we obtain
\begin{align}
    \int_{B_R} \langle  A_\varepsilon (x,Du_\varepsilon),D(u_\varepsilon-u) \rangle dx= \int_{B_R} f_\varepsilon (u_\varepsilon -u) dx,
\end{align}
and
\begin{align}
    \int_{B_R} \langle  A (x,Du),D(u_\varepsilon-u) \rangle  dx= \int_{B_R} f (u_\varepsilon -u) dx.
\end{align}
Subtracting the previous two equalities, we derive
 \begin{align}
   I= \int_{B_R} \langle  A_\varepsilon (x,Du_\varepsilon)- A(x, Du) \,, \, D(u_\varepsilon-u) \rangle dx= \int_{B_R} \left( f_\varepsilon- f \right) (u_\varepsilon -u) dx. \label{stima app}
\end{align}
We can write the integral on the left-hand side of \eqref{stima app} as follows
 \begin{align}
   I= \int_{B_R} \langle  A_\varepsilon (x,Du_\varepsilon)- A_\varepsilon(x, Du) \,, \, Du_\varepsilon-Du \rangle dx +  \int_{B_R} \langle  A_\varepsilon (x,Du)- A(x, Du) \,, \, Du_\varepsilon-Du \rangle dx,
\end{align}
thus getting
\begin{align}
    \int_{B_R} \langle  A_\varepsilon (x,Du_\varepsilon)- A_\varepsilon(x, Du) \,, \, Du_\varepsilon-Du \rangle dx = & \int_{B_R} \left( f_\varepsilon- f \right) (u_\varepsilon -u) dx \notag\\
    - & \int_{B_R} \langle  A_\varepsilon (x,Du)- A(x, Du) \,, \, Du_\varepsilon-Du \rangle dx.
\end{align}
Then, by assumption \eqref{A1} and H\"older's inequality, we infer
\begin{align}
    \nu \int_{B_R} \lvert Du_\varepsilon - Du \rvert^2 dx & \leq \int_{B_R} \langle  A_\varepsilon (x,Du_\varepsilon)- A_\varepsilon(x, Du) \,, \, Du_\varepsilon-Du \rangle dx \notag \\
    & \leq \int_{B_R} \lvert f_\varepsilon - f\rvert \, \lvert u_\varepsilon - u \rvert dx  \notag \\
    & \qquad + \int_{B_R} \lvert  A_\varepsilon (x,Du)- A(x, Du) \rvert \, \lvert Du_\varepsilon-Du \rvert dx \notag \\
    & \leq \left( \int_{B_R} \lvert f_\varepsilon - f\rvert^{\frac{2n}{n+2}}dx\right)^{\frac{n+2}{2n}}\left( \int_{B_R} \lvert u_\varepsilon - u\rvert^{\frac{2n}{n-2}}dx\right)^{\frac{n-2}{2n}} \notag \\
    & \qquad +\left( \int_{B_R}\lvert  A_\varepsilon (x,Du)- A(x, Du) \rvert^2 dx\right)^{\frac{1}{2}} \left( \int_{B_R}\lvert Du_\varepsilon-Du \rvert^2 dx\right)^{\frac{1}{2}} \notag \\
 & \leq c \left( \int_{B_R}\lvert Du_\varepsilon-Du \rvert^2 dx\right)^{\frac{1}{2}}
    \Biggl\{\left( \int_{B_R} \lvert f_\varepsilon - f\rvert^{\frac{2n}{n+2}}dx\right)^{\frac{n+2}{2n}} \notag \\
    & \qquad +\left( \int_{B_R}\lvert  A_\varepsilon (x,Du)- A(x, Du) \rvert^2 dx\right)^{\frac{1}{2}} \Biggr\}, \notag
\end{align}
where in the last line we used Sobolev-Poincar\'e inequality.
 Next, we divide both sides by $$\left( \int_{B_R}\lvert Du_\varepsilon-Du \rvert^2 dx\right)^{\frac{1}{2}}$$ and square the resulting inequality. This yields 
\begin{align}
    \int_{B_R}\lvert Du_\varepsilon-Du \rvert^2 dx &\leq 
    c\left( \int_{B_R} \lvert f_\varepsilon - f\rvert^{\frac{2n}{n+2}}dx\right)^{\frac{n+2}{n}}  +c \int_{B_R}\lvert  A_\varepsilon (x,Du)- A(x, Du) \rvert^2 dx  \notag \\
    & \leq  c \red \left( \int_{B_R} \lvert f_\varepsilon - f\rvert^{\frac{n}{ \alpha}}dx\right)^{\frac{2\alpha}{n}} \nc + c\int_{B_R}\lvert  A_\varepsilon (x,Du)- A(x, Du) \rvert^2 dx , \label{conv grad} 
\end{align}
 where we used that $\frac{2n}{n+2}< \frac{n}{\alpha}$.  By assumption \eqref{A2} and the fact that $A_\varepsilon(x,0) \in L^2(B_R)$, we have
\begin{align}
    \lvert A_\varepsilon (x, Du) \rvert &\leq \lvert A_\varepsilon (x, Du)- A_\varepsilon (x,0)\rvert + \lvert A_\varepsilon (x, 0)\rvert \notag \\
    & \leq L|Du| +|A_\varepsilon (x, 0)|   \in L^2(B_R) . \notag
\end{align}
Since 
$$A_\varepsilon (x, 0) \rightarrow A(x, 0) \quad  \text{in } L^2(B_R) \text{ as } \varepsilon \to 0$$
and
$$ A_\varepsilon (x, Du) \rightarrow A(x,Du) \quad \text{a.e.\ in } B_R \text{ as } \varepsilon \to 0,$$
by Lebesgue dominated convergence Theorem, we get
$$A_\varepsilon (x, Du) \rightarrow A(x,Du) \quad  \text{in } L^2(B_R) \text{ as } \varepsilon \to 0.$$
Therefore, letting $\varepsilon \to 0$ in \eqref{conv grad}, we have that 
$$u_\varepsilon \to u \quad \text{strongly in } W^{1,2}_0(B_R).$$
Now, $u_\varepsilon$ satisfies  inequalities \eqref{Linfty bound} and \eqref{stima ite}, i.e.
\begin{equation}
    \red  \Vert u \Vert_{L^\infty(B_\rho(x_0))} \le \, cR^\frac{2-\alpha}{2} \Vert f \Vert_{L^\frac{n}{\alpha}(B_R(x_0))}+\dfrac{c R^\frac{2n}{2-\alpha}}{(R-\rho)^{2n/(2-\alpha)}} \left( \fint_{B_R(x_0)}u^2 dx  \right)^\frac{1}{2} \nc \notag
\end{equation} 
and
\begin{align}
  \left(  \int_{ B_{R/2} } |Du_\varepsilon|^{q} dx \right)^{\frac{1}{q}}  & \leq  \tilde{c}\Biggl\{ \Vert u_\varepsilon\Vert_{L^\infty(B_R)}+ \red \Vert f_\varepsilon \Vert_{L^{\frac{n}{\alpha}}(B_R)} \nc+  \left( \int_{B_{R}} |Du_\varepsilon|^{2} dx\right)^{\frac{1}{2}}  + 1    \Biggr\}. \notag
\end{align}
Thus, we have 
\begin{align}
  \left(  \int_{ B_{R/2} } |Du_\varepsilon|^{q} dx \right)^{\frac{1}{q}}  & \leq  \tilde{c}\left(\red \Vert f_\varepsilon \Vert_{L^{\frac{n}{\alpha}}(B_{2R})} \nc+   \Vert u_\varepsilon \Vert_{W^{1,2}(B_{2R})}  + 1  \right),\notag
\end{align}
where $\tilde{c} $ is a positive constant independent of $\varepsilon$.
Eventually, letting $\varepsilon \to 0$ and  using a covering argument, we conclude that $u \in W^{1,q}_{\textrm{loc}}(\Omega)$. 
\endproof

\section{Example}\label{sec ex}
\noindent
In this section, we exhibit an example showing that Theorem~1.2 is sharp, in the sense that under assumption~(A3), one cannot, in general, expect solutions to~(1.1)
to be locally Lipschitz continuous.   More precisely,
we provide an equation satisfying assumption~(A3) whose solution satisfies
\[
Du \in L^q_{\mathrm{loc}} \quad \text{for every } q<\infty,
\qquad\text{but}\qquad
Du \notin L^\infty_{\mathrm{loc}}.
\]
We recall the following
\begin{prop}
 Setting for every $r \in (0,1)$ and a fixed $r_0>1$
$$v(r)=\log\frac{r_0}{r},$$
the function $u(x)=x_1v(|x|)$
 solves the equation
$$ - \ \mathrm{ div}(A(x)Du)=0 \qquad \text{in } B_1(0), $$
where the matrix $A(x)= (a_{ij}(x))_{\substack{1\leq i \leq n \\ 1 \leq j \leq n}}$ is defined by
\begin{equation}
    a_{ij}(x)= \delta_{ij} + \gamma(|x|) \left(\delta_{ij} - \frac{x_i x_j}{|x|^2} \right), \qquad \gamma(r)= \frac{-n}{(n-1)\log\frac{r_0}{r}}. \label{Matrice}
\end{equation}   
Moreover, $u$ is such that $Du \in L^q_{\mathrm{loc}}(B_1(0))$ for every $q>1$, but $Du \notin L^\infty_{\mathrm{loc}}(B_1(0))$.
\end{prop} 
See \cite[Proposition 1.5]{jin} for the proof.   \\

In order to prove that the operator $A(x,\xi)=A(x) \xi$, with $A(x)$ defined in \eqref{Matrice}, satisfies the assumption \eqref{A3}, it is sufficient to show that 
\begin{equation}
    A(x) \in W^{1,n}_{\textrm{loc}}(B_1(0)) ,\label{bes reg}
\end{equation}
 since by \cite[Remark 3, Section 2.7.1]{Triebel} and Lemma \ref{3.2}, it holds
%\begin{align}
%  W^{1,n}_{\textrm{loc}}(B_1(0)) \subset   B^\beta_{\frac{n}{ \beta},\infty,\textrm{loc}}(B_1(0)) \subset   B^\alpha_{\frac{n}{ \alpha},p,\textrm{loc}}(B_1(0)) , \notag
%\end{align}
%\textcolor{orange}{quella di sopra \'e vera, ma forse conviene scrivere questa? perch\'e la prima inclusione \'e quella nel libro di Triebel}
\begin{align}
  W^{1,n}_{\textrm{loc}}(B_1(0)) \subset   B^\beta_{\frac{n}{ \beta},n,\textrm{loc}}(B_1(0)) \subset   B^\alpha_{\frac{n}{ \alpha},p,\textrm{loc}}(B_1(0)) , \notag
\end{align}
for every $0 < \alpha < \beta <1$. 
\\By virtue of the pointwise characterization of Sobolev spaces in \cite{haj}, \eqref{bes reg} is equivalent to saying that
$$\lvert A(x)- A(y)\rvert \leq |x-y| \left( h(x) + h(y) \right),$$
for some non-negative function $h \in L^n_{\mathrm{loc}}(B_1(0))$ and for all $x,y \in B_1(0)$.
 
From the definition of $a_{ij}(x)$, we have
\begin{align}
    \lvert A (x) - A (y) \rvert &= \sum_{i,j}\lvert a_{ij}(x) - a_{ij}(y) \rvert \notag \\
    & = \sum_{i,j} \left\lvert \left( \gamma(|x|)- \gamma(|y|) \right) \delta_{ij} - \gamma(|x|)\frac{x_i x_j}{|x|^2}+ \gamma(|y|)\frac{y_i y_j}{|y|^2} \right\rvert \notag \\
    &\leq n\left\lvert  \gamma(|x|)- \gamma(|y|) \right\rvert +  \sum_{i,j}   \left\lvert \gamma(|x|)\frac{x_i x_j}{|x|^2}- \gamma(|y|)\frac{y_i y_j}{|y|^2} \right\rvert \notag \\
    &=: n I_1 + I_2. 
\end{align}
Without loss of generality, we assume $|x|> |y|$. From \eqref{Matrice}, we obtain 
\begin{align}
  I_1 &= \frac{n}{n-1} \left\lvert  - \frac{1}{\log\frac{r_0}{|x|}} + \frac{1}{\log\frac{r_0}{|y|}} \right\rvert  \notag \\
  &=  \frac{n}{n-1} \left\lvert  \frac{\log|y| - \log|x|}{\log\frac{r_0}{|x|} \log\frac{r_0}{|y|}} \right\rvert. \notag
  \end{align}
  Since
  \begin{align}
  \lvert \log |y| - \log|x| \rvert&=   \int_{|y|}^{|x|} \frac{1}{t} \,dt \leq  \int_{|y|}^{|x|} \frac{1}{|y|}\, dt =  \frac{1}{|y|} (|x| - |y|)\notag \\
  &\leq (|x| - |y|) \left( \frac{1}{|y|}+ \frac{1}{|x|} \right) \notag \\
   &\leq |x- y| \left( \frac{1}{|y|}+ \frac{1}{|x|} \right), \notag
\end{align}
where we used the fact that $\frac{1}{t} \leq \frac{1}{|y|}$. So, we have
\begin{align}\label{EsI1}
    I_1 \leq& \, \frac{n}{n-1} |x- y| \left( \frac{1}{|x|}+ \frac{1}{|y|}\right)  \frac{1}{\log\frac{r_0}{|x|} \log\frac{r_0}{|y|}}  \notag\\
    \leq & \, \frac{n}{n-1} |x- y| \left( \frac{1}{|x| \,\log\frac{r_0}{|x|} \log\frac{r_0}{|y|}} +  \frac{1}{|y| \,\log\frac{r_0}{|x|} \log\frac{r_0}{|y|}} \right) \notag\\
    \leq & \, c(n,r_0) |x- y| \left( \frac{1}{|x| \,\log\frac{r_0}{|x|} } +  \frac{1}{|y| \,\log\frac{r_0}{|y|}} \right) ,
\end{align}
where we used that $\frac{1}{\log \frac{r_0}{|x|}}$ is bounded in $B_1(0)$.\\
We take care of $I_2$
\begin{align}\label{EsI2}
    I_2 &= \left\lvert \gamma(|x|)\frac{x_i x_j}{|x|^2} - \gamma(|y|)\frac{x_i x_j}{|x|^2}+ \gamma(|y|)\frac{x_i x_j}{|x|^2} - \gamma(|y|)\frac{y_i y_j}{|y|^2} \right\rvert \notag \\
    &\leq  \left\lvert \gamma(|x|) -\gamma(|y|) \right\rvert \frac{|x_i x_j|}{|x|^2} +
    |\gamma(|y|)| \left\lvert \frac{x_i x_j}{|x|^2} - \frac{y_i y_j}{|y|^2} \right\rvert \notag \\
   & =: I'_{2} + I''_{2}.
\end{align}
The term $I'_2$ can be estimated as follows
\begin{align}\label{Es I2'}
    I_2' \leq  \left\lvert \gamma(|x|) -\gamma(|y|) \right\rvert \left( \frac{x_i^2}{2}+\frac{x_j^2}{2} \right)\frac{1}{|x|^2} \leq \left\lvert \gamma(|x|) -\gamma(|y|) \right\rvert.
\end{align}
On the other hand, for $I_2^{''}$ we have
\begin{align}\label{Es I2''}
    I_2'' &= | \gamma(|y|) | \, \frac{\left|x_ix_j |y|^2- y_iy_j|y|^2 + y_iy_j |y|^2 - y_iy_j|x|^2 \right|}{|x|^2 |y|^2} \notag \\
    &\leq | \gamma(|y|) | \, \frac{\left\lvert x_ix_j - y_iy_j \right\rvert |y|^2 +|y_iy_j|\left(|x|^2 -|y|^2\right)}{|x|^2 |y|^2} \notag \\
    &\leq   | \gamma(|y|) | \, \left( \frac{\left\lvert x_ix_j - y_iy_j \right\rvert }{|x|^2 }+\frac{|y_iy_j | \left( |x|^2- |y|^2 \right)}{|x|^2 |y|^2} \right) \notag\\
    &\leq    | \gamma(|y|) | \, \left( \frac{\left\lvert x_ix_j -x_iy_j+x_iy_j- y_iy_j \right\rvert }{|x|^2 }+\left(   \frac{y_i^2}{2}+\frac{y_j^2}{2}\right)\frac{\left(|y|- |x| \right) \left( |y| +|x| \right)}{|x|^2 |y|^2} \right) \notag \\
    & \leq  | \gamma(|y|) | \, \left( \frac{|x_j| |x_i -y_j|+|y_j||x_i- y_i| }{|x|^2 }+ \frac{1}{2}\frac{\left(|y|- |x| \right) \left( |y| +|x| \right)}{|x|^2} \right) \notag \\
    & \leq  | \gamma(|y|) | \, \left( \frac{ |x -y|(|x|+ |y|) }{|x|^2 }+\frac{1}{2} |y-x| \left( \frac{|y|}{|x|^2} + \frac{1}{|x|}\right)      \right) \notag\\
    & \leq \frac{c(n)}{\log \frac{r_0}{|y|}} \, |y-x| \left( \frac{|y|}{|x|^2} + \frac{1}{|x|}\right)   \notag\\
    & \le c(n) |y-x|\left( \frac{|y|}{|x|^2 \, \log \frac{r_0}{|y|}} + \frac{1}{|x|\,\log \frac{r_0}{|y|}}\right) \notag\\
    & \le c(n) |y-x|   \left( \frac{1}{|x| \,\log\frac{r_0}{|x|} } +  \frac{1}{|y| \,\log\frac{r_0}{|y|}} \right) ,
\end{align}
where we used that $|x|>|y|$. Inserting \eqref{Es I2'} and \eqref{Es I2''} in \eqref{EsI2}, we infer
\begin{align}\label{EsI2new}
    I_2 &\leq \left\lvert \gamma(|x|) -\gamma(|y|) \right\rvert + c(n) |y-x|   \left( \frac{1}{|x| \,\log\frac{r_0}{|x|} } +  \frac{1}{|y| \,\log\frac{r_0}{|y|}} \right) \notag \\
    & \le c(n) |y-x|   \left( \frac{1}{|x| \,\log\frac{r_0}{|x|} } +  \frac{1}{|y| \,\log\frac{r_0}{|y|}} \right),
\end{align}
where the first term of r.h.s. has been estimated as \eqref{EsI1}.
Combining \eqref{EsI1} and \eqref{EsI2new}, we find
\begin{align}
     \lvert A (x) - A (y) \rvert  & \leq c(n,r_0) |x- y| \left( \frac{1}{|x| \,\log\frac{r_0}{|x|} } +  \frac{1}{|y| \,\log\frac{r_0}{|y|}} \right).
\end{align}
In order to obtain assumption \eqref{A3}, it is sufficient to choose
$$h(x) := \frac{1}{|x| \log\frac{r_0}{|x|}}.$$
\\It remains to check whether $\frac{1}{|x|\log\frac{r_0}{|x|}} \in L^{n}(B_1(0))$. We have that
\begin{align}
    \int_{B_1(0)} \left( \frac{1}{|x|\log\frac{r_0}{|x|}} \right)^n dx & = c(n) \int_0^1 \frac{1}{\rho  \log^{n} \frac{r_0}{\rho}} \, d \rho < + \infty.\notag
\end{align}

%From the definition of $A(x)$ and the assumption \eqref{A2}, we obtain for every $x,y \in B_1(0)$, $x \neq y$,
%$$\lvert A(x)- A(y)\rvert \leq \frac{c}{\log\frac{r_0}{|x-y|}}.$$
%Therefore, we have that
%$$\lvert A(x)- A(y)\rvert  \leq \frac{c}{\log\frac{r_0}{|x-y|}}\leq|x-y|^\beta \left( h(x) + h(y) \right)$$
%holds true if 
%$$h(x)+h(y) \geq \frac{1}{|x-y|^\beta \log\frac{r_0}{|x-y|}} \qquad \text{for some } \beta > \alpha.$$
%It is sufficient to choose
%$$h(x) := \frac{1}{|x|^\beta \log\frac{r_0}{|x|}},$$
%with $\beta >\alpha$. 
%\\It remains to check whether $\frac{1}{|x|^\beta\log\frac{r_0}{|x|}} \in L^{\frac{n}{\beta}}(B_1(0))$. We have that
%\begin{align}
%    \int_{B_1(0)} \left( \frac{1}{|x|^\beta\log\frac{r_0}{|x|}} \right)^\frac{n}{\beta} dx & = c(n) \int_0^1 \frac{1}{\rho  \log^{\frac{n}{\beta}} \frac{r_0}{\rho}} \, d \rho \notag
%\end{align}
%is finite for $\beta < n$.

\vskip2cm

\noindent
\textbf{Acknowledgments.} The authors are members of the Gruppo Nazionale per l’Analisi Matematica,
la Probabilità e le loro Applicazioni (GNAMPA) of the Istituto Nazionale di Alta Matematica (INdAM). The authors have been partially supported through the INdAM - GNAMPA 2026 Project ``Esistenza e regolarità per soluzioni di equazioni ellittiche e paraboliche anisotrope'' (CUP: E53C25002010001). In addition S. Russo has also been supported through the project: Sustainable Mobility Center (Centro Nazionale per la Mobilità Sostenibile – CNMS) - SPOKE 10, grant number E63C22000930007.
\bigskip

\par\noindent {\bf Data availability statement.} Data sharing not applicable to this article as no datasets were generated or analysed during the current study.

\bigskip
\par\noindent
{\bf Conflict of Interest}. The authors declare that they have no conflict of interest.

\end{document}